\documentclass[10pt,a4paper]{article}
\usepackage{amsmath}
\usepackage{amsthm}
\usepackage{amssymb}
\usepackage{mathrsfs}
\usepackage{color}
\usepackage{eucal}

\definecolor{ddmagenta}{rgb}{0.7,0,1.0}
\definecolor{ddcyan}{rgb}{0,0.1,1.0}
\definecolor{dred}{rgb}{.8,0,0}
\definecolor{ddgreen}{rgb}{0,0.4,0.4}

\newcommand{\RR}{\mathbb{R}}

\newcommand{\bele}{\begin{lemm}}
\newcommand{\enle}{\end{lemm}}
\newcommand{\bedef}{\begin{defi}}
\newcommand{\bete}{\begin{teor}}
\newcommand{\eddef}{\end{defi}}
\newcommand{\ente}{\end{teor}}
\newcommand{\beos}{\begin{osse}}
\newcommand{\eddos}{\end{osse}}
\newcommand{\bepr}{\begin{prop}}
\newcommand{\empr}{\end{prop}}
\newcommand{\bepro}{\begin{prob}}
\newcommand{\empro}{\end{prob}}
\newcommand{\bede}{\begin{defin}}
\newcommand{\edde}{\end{defin}}
\newcommand{\beco}{\begin{coro}}
\newcommand{\enco}{\end{coro}}

\newcommand{\beeq}[1]{\begin{equation}
 \label{#1}}
\newcommand{\eddeq}{\end{equation}}

\newcommand{\beeqa}[1]{\begin{eqnarray}
  \label{#1}}
\newcommand{\eddeqa}{\end{eqnarray}}

\newcommand{\beal}[1]{\begin{align}
 \label{#1}}
\newcommand{\eddal}{\end{align}}

\newcommand{\bespl}[1]{\begin{split}
 \label{#1}}
\newcommand{\edspl}{\end{split}}

\newcommand{\bega}[1]{\begin{gather}
 \label{#1}}
\newcommand{\edga}{\end{gather}}

\newcommand{\beeqax}{\begin{eqnarray*}}
\newcommand{\eddeqax}{\end{eqnarray*}}

\newcommand{\teta}{\vartheta}

\newcommand{\dt}{\partial_t}

\newcommand{\uu}{{\bf u}}

\newcommand{\io}{\int_\Omega}

\newcommand{\eps}{\varepsilon}

\newcommand{\weak}{\rightharpoonup}
\newcommand{\weakstar}{\mathop{\rightharpoonup}^{*}}

\DeclareMathOperator{\sign}{sign}

 \DeclareMathOperator{\semidist}{e}
\DeclareMathOperator{\distance}{d}

\let\TeXchi\chi
\def\chi{{\setbox0 \hbox{\mathsurround0pt
$\TeXchi$}\hbox{\raise\dp0 \copy0 }}}

\newtheorem{maintheorem}{Theorem}
\newtheorem{theorem}{Theorem}[section]
\newtheorem{lemma}{Lemma}[section]
\newtheorem{proposition}[lemma]{Proposition}
\newtheorem{definition}[lemma]{Definition}%
\newtheorem{notation}[lemma]{Notation}%

\newtheorem{remark}[lemma]{Remark}%

\newtheorem{mainproblem}{Problem}

\begin{document}

\newcommand{\Q}{\Omega \times (0,T)}
\newcommand{\Qt}{\Omega \times (0,t)}
\newcommand{\Qtau}{\Omega \times (0,\tau)}
\newcommand{\til}{\widetilde}
\renewcommand{\part}{\partial_t}
\renewcommand{\teta}{\theta}
\newcommand{\accaunoz}{H^1_0(\Omega)}
\newcommand{\accadue}{H^2(\Omega)}
\newcommand{\nnu}{\nonumber}
\newcommand{\irre}{\rho}
\newcommand{\vinc}{\beta}
\newcommand{\weaksto}{{\stackrel{*}{\rightharpoonup}\,}}
\newcommand{\weakto}{\rightharpoonup}
\newcommand{\debole}{\,\weak\,}
\newcommand{\debolestar}{\,\weakstar\,}
\newcommand{\pairing}[4]{ \sideset{_{#1 }}{_{ #2}}  {\mathop{\langle #3 , #4  \rangle}}}
\newcommand{\comments}[1]{\marginpar{\scriptsize\textit{#1}}}

 \def\fin{\hfill
         \trait .3 5 0
         \trait 5 .3 0
         \kern-5pt
         \trait 5 5 -4.7
         \trait 0.3 5 0
 \medskip}
 \def\trait #1 #2 #3 {\vrule width #1pt height #2pt depth #3pt}
\newcommand{\forae}{\text{for a.a.}}
\newcommand{\aein}{\text{a.e.\ in}}

\newcommand{\down}{\searrow}
\newcommand{\up}{\to}

\newcommand{\R}{\mathbb{R}}
\newcommand{\N}{\mathbb{N}}
\newcommand{\Z}{\mathbb{Z}}
\newcommand{\C}{\mathbb{C}}
\newcommand{\M}{\mathbb{M}}
\newcommand{\F}{\mathbb{F}}

\newcommand{\piecewiseConstant}[2]{\overline{#1}_{\kern-1pt#2}}
\newcommand{\pwC}{\piecewiseConstant}
\newcommand{\underlinepiecewiseConstant}[2]{\underline{#1}_{\kern-1pt#2}}
\newcommand{\upwC}{\underlinepiecewiseConstant}

\newcommand{\piecewiseLinear}[2]{{#1}_{\kern-1pt#2}}
\newcommand{\pwL}{\piecewiseLinear}
\newcommand{\pwM}[2]{\widetilde{#1}_{\kern-1pt#2}}
 \def\trait #1 #2 #3 {\vrule width #1pt height #2pt depth #3pt}
\newcommand{\pwN}[2]{#1_{\kern-1pt#2}}
 \def\trait #1 #2 #3 {\vrule width #1pt height #2pt depth #3pt}
\newcommand{\la}{\langle}
\newcommand{\ra}{\rangle}
\newcommand{\uk}{\pwN {\uu^k}{\tau}}
\newcommand{\un}{\pwN {\uu^n}{\tau}}
\newcommand{\Fn}{\pwN {J^n}{\tau}}
\newcommand{\uku}{\pwN {\uu^{k-1}}{\tau}}
\newcommand{\unu}{\pwN {\uu^{n-1}}{\tau}}
\newcommand{\chin}{\pwN {\chi^n}{\tau}}
\newcommand{\Fun}{\pwN {\mathbf{F}^n}{\tau}}
\newcommand{\gun}{\pwN {\mathbf{g}^n}{\tau}}
\newcommand{\chinuno}{\pwN {\chi^{n+1}}{\tau}}
\newcommand{\dom}{\text{dom}}

\newcommand{\ue}{\uu_{\varepsilon}}
\newcommand{\uet}{\uu_{\varepsilon t}}
\newcommand{\uek}{\uu_{\varepsilon_k}}
\newcommand{\uekt}{\uu_{\varepsilon_{k} t}}

\newcommand{\mo}{m}
\newcommand{\mw}{\mathcal{Z}}
\newcommand{\mn}{\mathcal{N}}
\newcommand{\mv}{\mathcal{V}}
\newcommand{\mh}{\mathcal{H}}
\newcommand{\mvi}{\mathcal{V'}}
\newcommand{\mwi}{\mathcal{Z'}}
\newcommand{\dd}{\mathrm{d}}
\newcommand{\h}{|_{H}}
\newcommand{\va}{\|_{V}}
\newcommand{\vi}{\|_{V'}}
\newcommand{\sobneg}[2]{\mathcal{W}^{#1,#2}(\Omega)}
\newcommand{\CC}{\mathrm{C}}
\newcommand{\uchi}{\underline{\chi}}
\newcommand{\uuu}{\underline{w}}
\newcommand{\uf}{\underline{G}}
\newcommand{\ug}{\overline{G}}
\newcommand{\uchio}{\underline{\chi}_{0}}

\newcommand{\ene}{\mathcal{E}}
\newcommand{\cx}{X}
\newcommand{\dcx}{\distance_X}
\newcommand{\ddd}{\, \mathrm{d}}
 \newcommand{\sfl}{\mathcal{S}}
 \newcommand{\geneset}{\mathcal{S}}
\newcommand{\att}{\mathcal{A}}
\newcommand{\rest}{Z(\sfl)}
\newcommand{\app}{\mu}
\newcommand{\chimd}{\chi}
\newcommand{\wmd}{w}

\newcommand{\quext}{\quad\text}
\newcommand{\qquext}{\qquad\text}
\newcommand{\chimud}{\chi_{\delta,M,\app}}
\newcommand{\wmud}{w_{\delta,M,\app}}
\newcommand{\foraa}{\forae}

\newenvironment{rcomm}{\color{dred}}{\color{black}}

\newenvironment{comm}{\color{ddgreen} \textsf{R:}\,}{\color{black}}
\newenvironment{BOH}{\color{red} \textsf{ATT!!!}\,}{\color{black}}
\newenvironment{todo}{\color{ddcyan}}{\color{black}}

\newenvironment{ric}{\color{ddcyan}}{\color{black}}
\newenvironment{prob}{\color{ddmagenta}}{\color{black}}

\title{\Large Analysis of the Cahn-Hilliard equation \\
with  a chemical potential dependent mobility\thanks{All authors
have been supported by the {project  \emph{Programma
Galileo, Universit\`a Italo-Francese/Projet Galil\'ee ``Modelli matematici in scienza
dei materiali/Mod\`eles math\'ematiques en science des
mat\'eriaux''.}} This paper was initiated during a stay of M.G.,
R.R., and G.S. in the \emph{Laboratoire
 de Math\'ematiques et Applications} (Universit\'e de Poitiers), whose
hospitality is gratefully acknowledged.}}
                           %
\author{Maurizio Grasselli\footnote{ \emph{Dipartimento di Matematica
``F.\ Brioschi'', Politecnico di Milano. Via Bonardi, 9. I--20133
Milano, Italy. Email: {\tt maurizio.grasselli\,@\,polimi.it}}}\,, \, Alain
Miranville\footnote{ \emph{Laboratoire de Math\'ematiques et
Applications--SP2MI, Universit\'e de Poitiers. Boulevard Marie et
Pierre Curie--T\'el\'eport 2. F--86962 Chasseneuil Futuroscope
Cedex, France.  Email: {\tt miranv\,@\,math.univ-poitiers.fr}}}\,,
\, Riccarda Rossi\footnote{ \emph{Dipartimento di Matematica,
Universit\`a di
 Brescia. Via Valotti 9. I--25133 Brescia, Italy.}
 E-mail: {\tt riccarda.rossi\,@\,ing.unibs.it}}\,, \,
 Giulio Schimperna\footnote{\emph{Dipartimento di Matematica ``F.\
Casorati'', Universit\`a di Pavia. Via Ferrata, 1. I--27100 Pavia,
Italy. Email: {\tt giusch04\,@\,unipv.it}}}
  }

\date{March 11th, 2010}
 \maketitle

 \numberwithin{equation}{section}

\begin{abstract}

The aim of  this paper is to study the well-posedness and the
existence of global attractors for  a family of Cahn-Hilliard
equations with  a  mobility  depending on the chemical potential.
 Such models arise from ge\-ne\-ra\-li\-za\-tions of the
(classical) Cahn-Hilliard equation due to \textsc{M.~E.~Gurtin}.

\end{abstract}

\section{Introduction}
\label{s:1} In this paper, we address the initial and boundary value
problem for the following {\em generalized Cahn-Hil\-liard equation}:
\begin{equation}
\label{e:1} \chi_t -\Delta \alpha\left(\delta \chi_t -\Delta \chi
+\phi(\chi) \right)=0 \qquad \text{in }\Omega \times (0,T),
\end{equation}
where $\delta \geq 0$,  $\Omega \subset \R^3$ is a bounded domain,
$T>0$ a finite time horizon, and $\alpha:\R\to \R $ a strictly
increasing function.

The classical Cahn-Hilliard equation reads
\[
\chi _t-\Delta w=0,\ \ w=-\Delta \chi +\phi (\chi ) \qquad \text{in
}\Omega \times (0,T),
\]
\noindent where $\chi $ is the order parameter (corresponding to a
density of atoms), $w$ is the chemical potential (defined as a variational
derivative of the free energy with respect to the order parameter), and
$\phi $ is the derivative of a double-well potential. This equation plays
an essential role in materials science and describes phase separation
processes in binary alloys (see, e.g., \cite{Cah, CahH, NC}).

 By considering a
mechanical version of the second law of thermodynamics and
introducing a new balance law for interactions at a microscopic
level, \textsc{M.~E.~Gurtin} proposed in \cite{Gu} the following
equations:
\[
\begin{cases}
\displaystyle{
 \chi_t-{\rm  div}(A(\chi ,\nabla \chi ,\chi _t,w)\nabla w)=0,}
 \\
 \displaystyle{
w=\delta (\chi ,\nabla \chi ,\chi _t,w)\chi _t-\Delta \chi +\phi
(\chi )}
\end{cases} \qquad \text{in } \Omega \times (0,T),
\]
\noindent Taking $\delta $ constant and $A=a(w)I$,  with $a: \R \to
\R$ a positive function,
 we then obtain an equation of the form \eqref{e:1}, in which $\alpha$ is some primitive of  the function
 $a$.

In the viscous case $\delta >0$,  such equations  have been studied
in \cite{rossi05, rossi06}. Therein, results on the well-posedness
and the existence of global attractors have been obtained.

Our main aim in this paper is to treat the case $\delta =0$. We also
consider the viscous case $\delta >0$ under different (and more
general) assumptions on $\alpha $ and $\phi $ from those in
\cite{rossi05, rossi06}.  In particular, we prove the existence of
solutions both in the non-viscous case $\delta =0 $ (cf.
Theorem~\ref{th:1}) and in the
 viscous case $\delta >0$ (see Theorem~\ref{th:2}). In the latter setting, under
more restrictive assumptions on the nonlinearities $\alpha$ and
$\phi$, we also obtain (cf. Theorem 3.1) well-posedness and
continuous dependence results for (the Cauchy problem for)
\eqref{e:1}. For $\delta>0$ we are also able
  to study
the asymptotic behavior of the system and establish the existence of
the global attractor (see Theorem~\ref{th:4}) in a quite general
frame of assumptions on $\alpha$ and $\phi$, which may allow for
non-uniqueness of solutions. That is why, for this long-time
analysis we rely on the   notion of generalized semiflows proposed
by \textsc{J.M.~Ball} in \cite{Ball97}, and on the extension given
in \cite{rossi-segatti-stefanelli08}. Finally, relying on the
short-trajectory approach developed in~\cite{malek-prazak}, we also
conclude  the existence of exponential attractors and, thus, of
finite-dimensional global attractors.  We recall that
 an exponential attractor is a compact and semi-invariant
set which has finite fractal dimension and attracts the trajectories
exponentially fast; note that the global attractor may attract the
trajectories at a slow (polynomial) rate (see, e.g.,
\cite{BabinVishik92, EFNT, MZH}).

This paper is organized as follows. In Section 2, we define our
notation and give some preliminary results. Then, in Section 3, we
state our main results, whose proofs are carried out in the
remaining sections. Finally, in Appendix, we introduce the
approximation scheme for our problem and justify the a priori
estimates (formally) developed throughout the paper.

\section{Preliminaries} \label{ss:2.1}
\paragraph{Notation and  functional setup.}
Throughout the paper, we consider a bounded domain $\Omega
\subset \R^3$, with sufficiently smooth boundary $\partial \Omega$,
and write $|\mathcal{O}|$ for the  Lebesgue measure of any
(measurable) subset $ \mathcal{O} \subset \Omega$.
 Furthermore,
   given a Banach space $B$, we denote by
$\Vert\cdot\Vert_B$   the norm in  $B$  and by
 $_{B'}\langle\cdot,\cdot\rangle_B$ the duality pairing
between  $B'$ and $B$. We use the notation
 $$
 H:=L^2 (\Omega),\quad
 V:=H^1(\Omega), \quad
 Z:=\left\{ v \in H^{2}(\Omega) \: : \: \partial_{n} v=0 \right\},
$$
and identify $ H $ with its dual space $ H' $, so that
 $
  Z \subset V \subset H \subset V'\subset Z'  $, with
  dense and compact  embeddings.
We denote by   $ \mathcal{H} $, $ \mathcal{V} $, $ \mathcal{Z}
$,    $ \mathcal{V'} $, and $ \mathcal{Z'} $ the subspaces of the
elements $v$ of $H$, $ V $, $Z$,    $ V' $, and $Z'$,
respectively, with zero mean value $ \mo(v) =\frac{1}{|\Omega|}
\pairing{Z'}{Z}{v}{1} $. We    consider
  the operator
  \begin{equation}
  \label{def:opA}
 A: V \rightarrow V',  \qquad   _{V'}\langle Au, v \rangle _V:=\int_{\Omega} \nabla u \cdot\nabla v  \quad \forall u, v \in V,
\end{equation}
    and note that  $ Au \in \mvi $ for every $ u \in V $.
     Indeed,  the restriction of $ A $ to $ \mv $ is an
  isomorphism,  so that we can introduce  its  inverse
   operator $
    \mn: \mvi \rightarrow \mv $.
 We recall the   relations
  \begin{eqnarray}
    &&    \label{Aa}
    \pairing{V'}{V}{Au}{\mn (v)} = \pairing{V'}{V}{v}{u}\quad \forall u \in V, \; \forall v \in \mvi,
    \\    &&    \label{Bb}    \pairing{V'}{V}{u}{\mn (v)}  =\int_{\Omega} \nabla(\mn (u)) \cdot\nabla(\mn (v)) \,  \dd x =
\pairing{V'}{V}{v}{\mn (u) } \quad \forall u, v \in \mvi,
\end{eqnarray} and that, on
    account of Poincar\'{e}'s inequality for  zero mean value
    functions, the following norms on $V$ and~$V'$:
    \begin{eqnarray}
    &&
  \| u \va^{2}:= \pairing{V'}{V}{Au}{u}  + \,\mo(u)^2 \quad \forall u \in V,
  \nonumber  \\
 &&
    \| v \vi^{2}:= \pairing{V'}{V}{v}{\mn(v -\mo(v))}  + \,\mo(v)^2\quad \forall v \in V',
\nonumber
    \end{eqnarray}
    are equivalent to the standard ones.
   It follows
from the above  formulae that
 \[ \| v \vi^2= \pairing{V'}{V}{v}
{\mn(v)}=  \|\mn(v) \va^2
 \quad \forall v \in \mvi.
\]

  It is well known that the operator $A$~\eqref{def:opA}
   extends to an operator (which will be
  denoted by the same symbol) $A: H \to \mathcal{Z'} $.
The inverse of the restriction of $A$ to
   $ \mathcal{H}$
  is the extension of
   $\mn$  to an operator $\mn: \mathcal{Z'} \to \mathcal{H}
  $. By means of the latter, we define the space
\begin{equation}
\label{e:sob-neg}
\begin{gathered}
 \sobneg{-2}{q}:= \left\{v \in \mathcal{Z'}\, : \
\mn (v) \in L^q (\Omega) \right\} \qquad \text{for a given $q>1$,}
\\
\text{with the norm $\| v \|_{\sobneg{-2}{q}}:=\| \mn (v)\|_{L^q
(\Omega)}$.}
\end{gathered}
\end{equation}
The following result shows that, for $q\in (2,6)$ (which is the
index range  relevant to the analysis to be developed in what
follows, cf.~\eqref{e:2.17}), the space $\sobneg{-2}{q}$ can be
identified with the dual of the space
\[
 \sobneg{2}{q'}= \left\{ z \in \mathcal{V}\, : \, Az \in
 L^{q'}(\Omega)\right\}\,,
\]
  $q'$ being the conjugate exponent of $q$. We
endow the latter space with the norm $ \|z \|_{\sobneg{2}{q'}}:=
\|Az\|_{L^{q'}(\Omega)}$, which is equivalent to the standard
$W^{2,q'}$-norm by the (generalized) Poincar\'e
inequality. 
\begin{lemma}
\label{le:sob-dual} For $q\in (2,6)$,
 the operator  $\mathrm{J}: \sobneg{-2}{q} \to
(\sobneg{2}{q'})' $ defined by
\begin{equation}
\label{def:operator}
\pairing{(\sobneg{2}{q'})'}{\sobneg{2}{q'}}{\mathrm{J}(v)}{z}:=
\pairing{V'}{V}{Az}{\mn(v)} \qquad \text{for all $z \in
\sobneg{2}{q'}$ and $v \in \sobneg{-2}{q}$}
\end{equation}
is an isomorphism.
\end{lemma}
\begin{proof}
We preliminarily note that, since $q \in (2,6)$, the conjugate
exponent $q'$ belongs to $(6/5,2)$ and, consequently, one has the following
embeddings:
\begin{equation}
\label{e:embeddings} \mathcal{H} \subset \mathcal{V'} \subset
\sobneg{-2}{q}, \qquad L^{q'}(\Omega) \subset \mathcal{V'}, \qquad
\mathcal{Z} \subset \sobneg{2}{q'}\,.
\end{equation}
Clearly, the operator $\mathrm{J}$ is well defined, linear, and
continuous, since, for all $z \in \sobneg{2}{q'}$ and $v \in
\sobneg{-2}{q}$,
\begin{equation}
\label{ineq-cont}
\left|\pairing{(\sobneg{2}{q'})'}{\sobneg{2}{q'}}{\mathrm{J}(v)}{z}
\right| \leq  \| Az\|_{L^{q'}(\Omega)} \|\mn(v) \|_{L^q(\Omega)}
\leq \|z \|_{\sobneg{2}{q'}} \| v\|_{\sobneg{-2}{q}}\,.
\end{equation}
Furthermore, for every $v \in \sobneg{-2}{q}$, one can  choose $z_v =
\mn (|\mn(v)|^{q-2}\mn(v) )$ (note that $z_v$ is well defined and
belongs to $\mathcal{V}$, since $|\mn(v)|^{q-2}\mn(v) \in
L^{q'}(\Omega) \subset \mathcal{V'}$ by the second
of~\eqref{e:embeddings}). Then,
\[
\begin{aligned}
& \pairing{(\sobneg{2}{q'})'}{\sobneg{2}{q'}}{\mathrm{J}(v)}{z_v}=
 \|\mn(v)\|_{L^q
(\Omega)}^{q},
\\ &
 \|Az_v\|_{L^{q'}(\Omega)}=
\||\mn(v)|^{q-2}\mn(v)\|_{L^{q'}(\Omega)}= \|\mn(v)\|_{L^q
(\Omega)}^{q-1},
\end{aligned}
\]
so that
\[
\|\mathrm{J}(v)\|_{(\sobneg{2}{q'})'} \geq
\frac{\left|\pairing{(\sobneg{2}{q'})'}{\sobneg{2}{q'}}{\mathrm{J}(v)}{z_v}
\right|}{\|z_v \|_{\sobneg{2}{q'}}} = \frac{\|\mn(v)\|_{L^q
(\Omega)}^{q}}{\|\mn(v)\|_{L^q (\Omega)}^{q-1}}=\|\mn(v)\|_{L^q
(\Omega)}= \|v\|_{\sobneg{-2}{q}}\,.
\]
In view of~\eqref{ineq-cont}, we conclude that $\mathrm{J}$ is an
isometry. In particular, it is injective and  the image
$\mathrm{J}(\sobneg{-2}{q} )$ is closed in $(\sobneg{2}{q'})'$. To
conclude that $\mathrm{J}$ is surjective, we will prove that
\begin{equation}
\label{image:dense}
 \text{$\mathrm{J}(\sobneg{-2}{q} )$  is dense  in
$(\sobneg{2}{q'})'$.}
\end{equation}
Indeed, let  $\bar{z}  \in \sobneg{2}{q'} $ be such that
\begin{equation}
\label{e:test-densita}
\pairing{(\sobneg{2}{q'})'}{\sobneg{2}{q'}}{\mathrm{J}(v)}{\bar{z}}=0
\ \ \text{for all $v \in \sobneg{-2}{q}$.}
\end{equation}
  In particular, \eqref{e:test-densita} holds for all $v \in
  \mathcal{H}$, so that, also in view of~\eqref{Aa},
  \[
0 = \pairing{V'}{V}{A\bar{z}}{\mn(v)} =\pairing{V'}{V}{v}{\bar{z}}
= \int_{\Omega} \bar{z} v \quad \text{for all  $v \in
\mathcal{H}$\,.}
  \]
From the above relation, we easily conclude that $\bar{z}=0$,
whence~\eqref{image:dense}.
\end{proof}
\paragraph{A generalization of Poincar\'e's inequality.}
The following result will play an important role in the derivation
of the \emph{a priori estimates} of  Section~\ref{s:3.1}.
\begin{lemma}
\label{le-poinca} Let $X$ and $Y$ be Banach spaces, with $X$
reflexive, and assume that
\begin{equation}
\label{e:compact-embed} \text{$X \Subset Y$ with compact
embedding.}
\end{equation}
 Consider
\begin{align}
& \label{funz-g} G: X \to Y \quad \text{a linear,
weakly-weakly continuous functional,}
\\
& \label{funz-psi}
\begin{aligned}
\Psi: X \to [0,+\infty)  \quad & \text{a $1$-positively
homogeneous,} \\
 & \text{sequentially weakly lower-semicontinuous functional.}
 \end{aligned}
\end{align}
Assume that $G$ and $\Psi$ comply with the following {\em
compatibility condition}: for all $v \in X$,
\begin{equation}
\label{e:comp} Gv=0 \quad \text{and} \quad \Psi(v) =0 \ \Rightarrow
\ v=0\,,
\end{equation}
and that
\begin{equation}
\label{e:norm-equivalent} \exists\, C\geq 1 : \ \ \forall\, v \in X
\quad \frac{1}C \left( \| v\|_{Y} + \| Gv\|_{Y} \right) \leq
\|v\|_{X} \leq C\left( \| v\|_{Y} + \| Gv\|_{Y} \right)\,.
\end{equation}
Then,
\begin{equation}
\label{e:poinc-gen} \exists \, K>0: \ \ \forall\, v\in X  \quad
\|v\|_{X} \leq K\left( \| Gv\|_{Y} + \Psi(v) \right)\,.
\end{equation}
\end{lemma}
\noindent {\em Proof.} \, Assume, by contradiction,
that~\eqref{e:poinc-gen} does not hold: then, there exists a
sequence $\{ v_n \} \subset X$ such that, for every $n \in \N$,
\begin{equation}
\label{e:contradd} \|v_n\|_{X} > n \left( \| Gv_n\|_{Y} + \Psi(v_n)
\right)\,.
\end{equation}
In particular, this yields that $\|v_n\|_{X} \neq 0$ for all $n$.
Letting $w_n:= v_n / \|v_n\|_{X}$ and using the $1$-homogeneity of
$\Psi$, we deduce from~\eqref{e:contradd} that
$$  \| Gw_n\|_{Y} +
\Psi(w_n) <\frac1n \quad \text{for every $n \in \N$}\,,
$$
giving
\begin{equation}
\label{e:2.1.12}\text{ $Gw_n \to 0$ in $Y$ \ and  \ $\Psi(w_n) \to
0$ as $n \to +\infty$.}
\end{equation}
 On the other hand, by the reflexivity of $X$, there exists a
subsequence $\{ w_{n_k}\}$ weakly converging in $X$ to some
$\bar{w}$. In view of~\eqref{e:compact-embed}--\eqref{funz-psi}, we
find
$$
w_{n_k} \to w \ \ \text{in $Y$}, \qquad Gw_{n_k} \weakto Gw \ \
\text{in $Y$}, \qquad \Psi(w) \leq \lim_{k \to +\infty}
\Psi(w_{n_k})\,.
$$
Hence, \eqref{e:2.1.12}  yields that $Gw=0$ and $\Psi(w)=0$, so
that, by \eqref{e:comp}, $w=0$. Thus, by \eqref{e:norm-equivalent}
and \eqref{e:2.1.12},
$$
\lim_{k \to +\infty}\|w_{n_k}\|_{X}\leq  C\lim_{k \to +\infty} \left(
\|w_{n_k} \|_{Y} + \| Gw_{n_k}\|_{Y} \right) =0,
$$
 in contrast
with the fact that $\| w_n\|_{X}=1$ for all $n \in \N$. \fin
\paragraph{A compactness criterion.}
 Let
\begin{equation}
\label{e:setting}
\begin{gathered}
\text{$\mathcal{O} \subset \R^d$, $d\geq 1$, be an open set with
$|\mathcal{O}|<+\infty$,}
\\
\text{ $B$ be a separable Banach space, and $ 1 \leq p <+\infty$.}
\end{gathered}
\end{equation}
We recall that a sequence  $ \{ u_n \} \subset L^{p}(\mathcal{O};B)
$ is \emph{$p$-uniformly integrable} (or simply \emph{uniformly
integrable} if $p=1$) if
\begin{equation}
  \label{lour}
  \forall\, \varepsilon > 0 \quad \exists\, \delta > 0:\quad
  \forall\, J
  \subset \mathcal{O} \quad |J| <
  \delta\ \Rightarrow\ \sup_{n \in \N} \int_{J}
  \|u_n(y) \|_{B}^{p} \ddd y  \leq \varepsilon.
\end{equation}
We quote the following result (cf.~\cite[Thm.~III.6]{Dunford-Schwartz58}) which will be
extensively used in what follows.
\begin{theorem}
\label{t:ds} In the setting of~\eqref{e:setting}, given a sequence $
\{ u_n \}
 \subset L^{p}(\mathcal{O};B) $, assume that there exist a  subsequence $\{u_{n_k}\}$
and a measurable function $u :\mathcal{O} \to B$
  such that
 \[
u_{n_k} (y) \to u(y) \ \  \text{in $B$} \ \ \text{for almost all}\
y \in \mathcal{O}\,.
 \]
Then, $u_{n_k} \to u$ in $L^{p}(\mathcal{O};B) $ if and only if it
is $ p $-uniformly integrable.
\end{theorem}
\noindent
 Finally, for the reader's convenience, here below we report the celebrated
 lower semi\-continuity result due to~\textsc{A.D.~Ioffe}~\cite{ioffe77}.
 \begin{theorem}
 \label{th-ioffe}
Let $f: \mathcal{O} \times \R^n \times \R^m \to [0,+\infty]$,
$n,\,m\geq 1$, be a measurable non-negative function such that
\begin{align}
\label{hyp:ioffe1} &
 f(x,\cdot,\cdot) \ \ \ \text{is lower
semicontinuous on $\R^n \times \R^m$ for every $x \in \mathcal{O}$,}
\\
& f(x,u,\cdot) \ \ \ \text{is convex on  $ \R^m$ for every $(x,u)
\in \mathcal{O} \times \R^n$.}
\end{align}
Let $(u_k,v_k), \ (u,v): \mathcal{O} \to \R^n \times \R^m $ be
 measurable functions such that
\[
u_k(x) \to u(x)  \quad \text{in measure in $ \mathcal{O}$,} \qquad
v_k \weakto v \quad \text{weakly in $L^1 ( \mathcal{O}; \R^m)$}.
\]
Then,
\begin{equation}
\label{integral-lsc} \liminf_{k \to +\infty}  \int_{\mathcal{O}}
f(x,u_k(x),v_k(x)) \, \mathrm{d}x \geq \int_{\mathcal{O}}
f(x,u(x),v(x)) \, \mathrm{d}x\,.
\end{equation}
\end{theorem}
\subsection{Global attractors for generalized semiflows}
\label{ss:2.1.1}
\noindent
 As mentioned in the introduction, in order to study the
long-time behavior of solutions to the generalized Cahn-Hilliard
equation~\eqref{e:1} \emph{in the viscous case},  we rely on the
theory of \emph{generalized} semiflows introduced by
\textsc{J.M.~Ball} in~\cite{Ball97}. In order to make this paper as
self-contained as possible, in this section we recall the main
definitions and results of this theory,
 closely
following  \cite{Ball97}.
\begin{notation}
\label{not:phase-space}
 The phase space is  a (not necessarily complete) metric space  $(\cx,
\dcx)$, the distance $\dcx$ inducing
 the \emph{Hausdorff semidistance}
$\semidist$ of two non-empty subsets $A, \, B \subset \cx$
 by the formula $\semidist(A,B):= \sup_{a \in A} \inf_{b \in B} \dcx(a,b)$. 
\end{notation}
\begin{definition}[Generalized semiflow]
\label{def:generalized-semiflow} A \emph{generalized semiflow}
$\sfl$ on $\cx$ is a family of maps $g:[0,+\infty) \to \cx$
(referred to as ``solutions") satisfying the following properties:
\begin{description}
\item[(P1)] \emph{\bf (Existence)} for any $g_0 \in \cx$, there exists at
least one $g \in \sfl$ such that  $g(0)=g_0$;
 \item[(P2)]
\emph{\bf (Translates of solutions are solutions)} for any $g \in
\sfl$ and $\tau \geq 0$, the map $g^\tau (t):=g(t+\tau),$ $t \in
[0,+\infty),$ belongs to $\sfl$;
 \item[(P3)]
\emph{\bf (Concatenation)} for any $ g $, $ h \in \sfl$ and $\tau
\geq 0$  with $h(0)=g(\tau)$, then  $z \in \sfl$, $ z$ being the map
defined by
\begin{equation}
\label{def:concaten}
 z(t):= \begin{cases} g(t)  & \text{if $0 \leq t \leq \tau,$}
 \\
 h(t-\tau) & \text{if $ t >\tau$;}
 \end{cases}
 \end{equation}
\item[(P4)] \emph{\bf (Upper-semicontinuity w.r.t. the initial data)} if
$\{g_n \} \subset \sfl$ and $g_n (0) \to g_0, $ then there exist a
subsequence $\{g_{n_k}\}$ of $\{g_n \}$ and $g \in \sfl$ such that $
g(0)=g_0$ and  $g_{n_k}(t) \to g(t)$ for all $t \geq 0.$
\end{description}
\end{definition}
\paragraph{Orbits, $\omega$-limits and attractors.}
 Given  a solution $ g \in \sfl$, we recall that
  the \emph{$\omega$-limit} $\omega(g)$
of $g$
  is
defined by
$$
\omega(g):= \{ x \in \cx \ : \ \exists \{t_n\} \subset [0,+\infty),
\ t_n \to +\infty, \ \text{such that} \ \ g(t_n) \to x \}\,.
$$
Similarly,  the \emph{$\omega$-limit} of a set $E \subset \cx$ is
given by
\[
\begin{aligned}
\omega(E):=\big\{
 x \in \cx \ : \ &  \exists  \{g_n\} \subset \geneset \
 \text{such that $\{g_n (0)\} \subset E$, $\{g_n (0)\}$ is bounded,
and} \\
 & \exists \{t_n\} \subset [0,+\infty), \ t_n \to +\infty,\
\text{such that $ g_n (t_n) \to x$} \big\}.\nonumber
\end{aligned}
\]
Furthermore,
 we say that $w: \R \to \cx$ is  a \emph{complete orbit}
if, for any $s \in \R$, the translate
 map $w^s$, restricted to  the positive half-line $[0,+\infty),$ belongs to $ \sfl$.
For every $\,t \geq 0 $,
 we can
 introduce the
operator $\,{T}(t): 2^\cx \to 2^\cx\,$~by setting
 \begin{equation}
 \label{eq:operat-T}
  {T}(t)E:=\{ g(t) \ : \
g \in \geneset \ \  \text{with} \ \ g(0) \in E\}\quad \text{for all
} E \subset \cx,
 \end{equation}
 and define, for $\tau \geq 0$, the set
 \[
\gamma^{\tau}(E):= \cup_{t \geq \tau} {T}(t)E\,.
 \]
 The family of operators $\{{T}(t)\}_{t \geq 0}$
  defines  a \emph{semigroup}  on
the power set $2^\cx$. Given subsets $U, E \subset \cx$, we say that
$U $ \emph{attracts} $E$ if $\semidist(T(t)E,U) \to 0$ as $t \to
+\infty$. Furthermore, we say that $U $ is  \emph{fully invariant}
if $T(t)U = U$ for every $t \geq 0$. Finally, a set $\mathcal{A}
\subset \cx$ is the  \emph{global attractor} for $\sfl$ iff  it is
compact, fully invariant under $\sfl$, and attracts all the bounded
sets of $\cx$.
\paragraph{\bf Compactness and dissipativity properties.}
 Let $ \geneset $ be a generalized semiflow. We say that $
 \geneset$ is
\begin{description}
\item \emph{\bf eventually bounded} iff, for every bounded set $B \subset
\cx$, there exists $\tau \geq 0$ such that $ \gamma^\tau (B)$ is
bounded;
\item \emph{\bf point dissipative} iff there exists a bounded
set $B_0 \subset \cx$ such that, for any $g \in \geneset$, there
exists $\tau \geq 0$ such that $g(t) \in B_0 $ for all $t \geq
\tau$.  The set $B_0$ is then called a (pointwise) \emph{absorbing} set;
\item \emph{\bf compact} iff, for any sequence $\{g_n \} \subset \geneset$
with $\{g_n (0)\}$  bounded,  there exists a subsequence $\{g_{n_k}
\} $ such that $ \{g_{n_k}(t) \} $ is convergent for any $t
>0.$
\end{description}
We note that the  notions that we have just  introduced are not
independent one from another (cf. \cite[Props. 3.1 and 3.2]{Ball97}
for more details).
\paragraph{\bf
Lyapunov function.}  The notion of a \emph{Lyapunov function} can be
introduced starting from the following definitions: we say that a
complete orbit $g \in \geneset$ is \emph{stationary} if there exists
$x \in \cx$ such that $g(t)=x$ for all $t \in \R$ - such an $x$ is
then called a \emph{rest point}.  Note that the set of rest points
of $\geneset $, denoted by $Z(\geneset)$, is closed in view of
\textbf{(P4)}. A function $V: \cx \to \R$ is said to be a
\emph{Lyapunov function} for $\geneset$ if $V$ is continuous,
$V(g(t)) \leq V(g(s))$ for all $g \in \geneset$ and $0 \leq s \leq
t$ (i.e., $V$ decreases along all solutions), and, whenever the map
$t \mapsto V(g(t))$ is constant for some complete orbit $ g$, then $
g $ is a stationary orbit.
\paragraph{\bf Existence of the global attractor.}
The following theorem subsumes the main results from~\cite{Ball97}
(cf. Thms. 3.3, 5.1, and 6.1 therein) and provides the basic
criteria for the existence of the global attractor $\att$ for a
generalized semiflow $\sfl$.
\begin{theorem}
\label{thm:ball1} Let $\sfl$ be  an eventually bounded and compact
 generalized semiflow. Assume that $\sfl$ also admits a Lyapunov function $V$ and that
\begin{equation}
\label{rest-bounded} \text{
 the set of its rest points $\rest$ is bounded.}
\end{equation}
  Then, $\sfl$ is
also point dissipative, and, consequently,  it possesses
 a global
attractor.
 Moreover, the attractor $\att$ is unique, it  is
the maximal compact fully invariant subset of $\cx$, and it can be
characterized as
\begin{equation}\label{eqn:attrattore}
\att= \bigcup \{\omega(B) \ : \ \text{$B \subset \cx$
bounded}\}=\omega(\cx).
\end{equation}
Finally, for every $g \in \sfl$,
\begin{equation}
\label{e:additional} \omega(g) \subset \rest.
\end{equation}
\end{theorem}
\begin{remark}
\label{rem:restriction_to_invariant_set} \upshape Actually, it is
immediate to check that, if $\geneset$ is compact, eventually
bounded, and admits a Lyapunov function, then
 condition~\eqref{rest-bounded}  can be
replaced by
\begin{equation}
\label{e:saab}
\begin{gathered}
 \exists\, \mathcal{D} \subset \cx\,, \ \  \mathcal{D}\neq
 \emptyset\,, \  \ \text{such that}
 \ \ \begin{cases}
 {T}(t) \mathcal{D} \subset \mathcal{D}
\quad \forall t \geq 0,
\\
 \text{the set $ Z(\geneset) \cap \mathcal{D} $ is bounded in
$\cx$}. \end{cases}
\end{gathered}
\end{equation}
 Then,  under these hypotheses, $\geneset$ also
possesses a (unique)  global attractor  $\att \subset
\mathcal{D}$ and~\eqref{e:additional} holds.
\end{remark}
\section{Main results}
\label{s:2}
\subsection{A global existence result for  the non-viscous problem}
\label{ss:2.2}
\paragraph{Assumptions on the nonlinearities.}
We assume that
\begin{equation}
\label{e:hyp1} \tag{H1}
\begin{gathered}
\alpha: \R \to \R \qquad \text{is a strictly increasing,
differentiable function such that}
\\
\exists\, p \geq 0, \ \ \exists\, C_1,\, C_2 >0 \, : \quad \forall\,
r \in \R \qquad C_1 \left( |r|^{2p}+ 1\right) \leq \alpha'(r) \leq
C_2 \left( |r|^{2p}+1\right)\,.
\end{gathered}
\end{equation}
Clearly,  the latter growth condition entails that
\begin{equation}
\label{e:conse1} \exists\, C_3,\, C_4,\, C_5 >0 \, : \quad \forall\,
r \in \R \qquad C_3  |r|^{2p+1}-C_4 \leq \alpha(r)\text{\rm sign}(r)
\leq C_5 \left( |r|^{2p+1}+1\right)\,.
\end{equation}

Concerning the nonlinearity $\phi$,
  we require that
\begin{equation}
\label{e:hyp2}
\tag{H2}
\begin{gathered}
\text{dom}(\phi) = I, \ \ \text{$I$ being an open, possibly unbounded,
interval $(a,b)$, $ -\infty \leq a < 0 < b \leq +\infty$,}
\\
\phi \in \CC^1 (I),
\\
\lim_{r \searrow a} \phi(r) = -\infty,\qquad \lim_{r \nearrow b}
\phi(r) = +\infty,
\\
\lim_{r \searrow a} \phi'(r) =\lim_{r \nearrow b} \phi'(r) =
+\infty\,.
\end{gathered}
\end{equation}
We shall denote by $\widehat{\phi} $ (one of) the antiderivative(s)
of $\phi$. It follows from the above assumptions that
$\widehat{\phi}$ is bounded from below.
 Hereafter, for the sake of
simplicity, we assume that
\begin{equation}
\label{phicappucciopos} \widehat{\phi}(r) \geq 0 \quad \text{for all
$r \in I$.}
\end{equation}
Furthermore,  \eqref{e:hyp2} obviously  yields that
\begin{equation}
\label{e:hyp2-bis}  \exists\, C_{\phi,1} >0 \, : \ \ \forall\, r \in
I \qquad \phi'(r) \geq -C_{\phi,1}\,,
\end{equation}
namely, $\phi$ is a Lipschitz perturbation of a non-decreasing
function. In particular, we will use the fact that there exists a
non-decreasing function $\beta: I \to \R$ such that
\begin{equation}
\label{e:hyp2-aftermath} \phi(r) = \beta(r) -C_{\phi,1} r \qquad
\forall\, r \in I\,.
\end{equation}
Consequently,  $\widehat{\phi}$ is a quadratic perturbation of a
convex function.
 Arguing in
the very same way as in~\cite{mirzel04} (where the case $I= (-1,1)$
was considered), it can be proved
  that, under these conditions, the
following crucial estimate holds:
\begin{equation}
\label{e:2.2.14} \forall\, m \in (a,b)\ \  \exists\, C_m,\ C_m'>0 \,
: \ \  \forall\, r \in (a-m,b-m) \quad |\phi(r+m)| \leq C_m
\phi(r+m)r +C_m'\,.
\end{equation}
\\
 Finally, we also assume that
\begin{equation}
\label{hyp:3} \tag{H3} \exists\, \sigma \in (0,1), \ \ \exists\,
C_6>0\, : \quad \forall\, r \in (a,b) \quad |\phi(r)|^{\sigma} \leq
C_6 \left( \widehat{\phi}(r) +1\right)\,,
\end{equation}
and that the following {\em compatibility condition} holds
between $\sigma$ and the  growth index $p$ of $\alpha$
in~\eqref{e:hyp1}:
\begin{equation}
\label{hyp:4} \tag{H4} \sigma >
\max\left\{\frac{6p-3}{6p+2},0\right\} \,.
\end{equation}
Hence,  if $p\leq 1/2$, then any $\sigma  \in  (0,1)$ is admissible,
while if, for instance,  $p=1$, then the range of admissible
$\sigma$'s is $(3/8, 1)$, and it is  $(9/{14},1)$ for $p=2$.
\noindent
\begin{notation}
\label{not:2.1} \upshape Hereafter, we will use, for every $p\geq 0$,
the short-hand notation
\begin{equation}
\label{e:index-notation} \rho_p:= \frac{2p+2}{2p+1}, \qquad
\kappa_p:=\frac{6p+6}{2p+1}, \qquad \eta_{p\sigma}=
\frac{6-\sigma}{(3-3\sigma)(2p+1)}\,.
\end{equation}
 For later convenience, we note that $\rho_p$ and $\kappa_p$ are
 decreasing functions of $p$ and
\begin{equation}
\label{e:2.17} 1<\rho_p<2, \qquad 3<\kappa_p <6 \qquad  \text{for
every $p\geq 0$.}
\end{equation}
Furthermore,  it can be checked that
\begin{equation}
\label{e:really-necessary} \eta_{p\sigma}>1 \quad \text{for every} \
p \geq 0  \ \text{and for all} \   \sigma >
\max\left\{\frac{6p-3}{6p+2},0\right\}\,.
\end{equation}
\end{notation}

\paragraph{The existence result.}
We are now able to give the variational formulation of the boundary
value problem associated with~\eqref{e:1} in the non-viscous case.
\begin{mainproblem}
\label{p:1}\upshape Find a pair $(\chi,w)$ fulfilling
\begin{align}
& \label{1-var} \chi_t + A (\alpha(w)) =0 \qquad \text{in
$\sobneg{-2}{\kappa_p}$} \quad \aein\ (0,T)\,,
\\
& \label{2-var} A\chi + \phi(\chi) =w \qquad \aein \ \Omega \times
(0,T)\,.
\end{align}
\end{mainproblem}
\noindent Note that, owing to  Lemma~\ref{le:sob-dual},
\eqref{1-var} is equivalent to
\begin{equation}
\label{e:difficult}
\begin{aligned}
\pairing{\sobneg{-2}{{\kappa_p}}}{\sobneg{2}{{\kappa_p}'}}{\chi_t}{v}
 & + \pairing{\sobneg{-2}{{\kappa_p}}}{\sobneg{2}{{\kappa_p}'}}{A
(\alpha(w))}{v}=0 \\ &   \text{for all $ v \in
\sobneg{2}{{\kappa_p}'} \quad \aein \ (0,T)$.}
\end{aligned}
\end{equation}
\begin{maintheorem}
\label{th:1} Under assumptions~\eqref{e:hyp1}--\eqref{hyp:4}, for
every initial datum $\chi_0$ satisfying
\begin{equation}
\label{hyp:initial-datum} \chi_0 \in V, \qquad
\widehat{\phi}(\chi_0) \in L^1(\Omega)\,,
\end{equation}
 there exists at least a
solution $(\chi,w)$ to Problem~\ref{p:1}, with the regularity
\begin{align}
& \label{reg-chi} \chi \in L^2 (0,T;W^{2,6} (\Omega)) \cap L^\infty
(0,T;V), \qquad \chi_t \in L^{\eta_{p\sigma}}(0,T;
\sobneg{-2}{\kappa_p})\,,
 \\
& \label{reg-w0} w \in L^2 (0,T;V), \qquad \alpha(w) \in
L^{\eta_{p\sigma}}(0,T;L^{\kappa_p}(\Omega))\,,
\end{align}
fulfilling the initial condition
\begin{equation}
\label{e:init-cond} \chi(0)=\chi_0 \quad \text{in $V$.}
\end{equation}
\end{maintheorem}
\noindent A   formal proof of this result  will be developed
 in Section~\ref{s:3} and rigorously justified in Appendix.
\subsection{A global existence result for the viscous problem}
\label{ss:2.3}
 We replace our
assumptions~\eqref{e:hyp2}--\eqref{hyp:4} on $\phi$ and its
antiderivative $\widehat{\phi}$ by
\begin{equation}
\label{hyp:5} \tag{H5}
\begin{gathered}
 \widehat{\phi} : \R \to \R \quad \text{belongs to $\CC^2 (\R)$ and satisfies}
\\
 \exists\, C_7>0\,
: \quad \forall\,r \in  \R  \quad |\phi(r)| \leq C_7 \left(
\widehat{\phi}(r) +1\right)\,.
\end{gathered}
\end{equation}
The latter assumption means that we consider potentials with at most
an exponential growth at $\infty$, and it clearly yields that
$\widehat{\phi} $ is bounded from below. Hence, as
in~\eqref{phicappucciopos},  we again assume that $\widehat{\phi}$
takes non-negative values.
 Furthermore,  as
in the non-viscous case  we require that
\begin{equation}
\label{e:hyp2-bis-visco} \tag{H6}
\begin{gathered}
 \exists\, C_{\phi,2} >0 \, : \ \ \forall\, r \in \R
\qquad \phi'(r) \geq -C_{\phi,2}\,.
\end{gathered}
\end{equation}
This and~\eqref{hyp:5} imply that \text{the map}
\begin{equation}
\label{e:lambda-convex}   \ r \in \R \mapsto \widehat{\phi}(r)
+\frac{C_{\phi,2}}2 r^2 \ \ \text{is convex and bounded from below.}
\end{equation}
\begin{remark}
\upshape Let us point out that~\eqref{e:lambda-convex} yields
\begin{equation}
\label{e:phi-add} |\widehat{\phi}(r)| \leq  |\widehat{\phi}(0)| +
|\phi(r)||r|+ \frac{C_{\phi,2}}{2}r^2    \quad \text{for all $r \in
\R$.}
\end{equation}
Indeed, it follows from~\eqref{e:lambda-convex} and an elementary
convexity inequality that, for every $r \in \R$,
\[
\widehat{\phi}(0) - \widehat{\phi}(r) - \frac{C_{\phi,2}}{2}r^2 \geq
-r\left( \phi(r) + C_{\phi,2}r \right)\,,
\]
whence we deduce~\eqref{e:phi-add} with straightforward algebraic
manipulations.
\end{remark}

We  will address  the analysis of the Cahn-Hilliard
equation~\eqref{e:1} in the viscous case under the aforementioned
assumptions. The related variational formulation reads
\begin{mainproblem}
\label{p:2}\upshape  Given $\delta >0$,  find a pair $(\chi,w)$
fulfilling
\begin{align}
& \label{1-var-better}
 \chi_t + A (\alpha(w)) =0 \qquad \aein \  \Omega \times (0,T)\,,
 \\
 &
 \label{2-var-better} \delta \chi_t+ A\chi + \phi(\chi) =w \qquad \aein \ \Omega \times
(0,T)\,.
 \end{align}
\end{mainproblem}
\paragraph{The existence result.}
\begin{maintheorem}
\label{th:2} Assume \eqref{e:hyp1}, \eqref{hyp:5},
and~\eqref{e:hyp2-bis-visco}. Then, for every initial datum $\chi_0$
complying with~\eqref{hyp:initial-datum}, there exists at least a
solution $(\chi,w)$ to Problem~\ref{p:2}, with the regularity
\begin{align}
& \label{reg-chi-bis} \chi \in L^2 (0,T;Z) \cap L^\infty (0,T;V)
\cap H^1 (0,T;H)\,,
 \\
& \label{reg-w} w \in L^2 (0,T;V) \cap
L^{2p+2}(0,T;L^\infty(\Omega)), \qquad \alpha(w) \in
L^{\rho_p}(0,T;Z)\,,
\end{align}
and such that $\chi$ satisfies the initial
condition~\eqref{e:init-cond}.
\end{maintheorem}
\noindent We refer to Section~\ref{s:3} for a formal proof of
Theorem~\ref{th:2} and to  Appendix for all rigorous calculations.

In addition, we also have the following regularity result, which
plays a key role in Section~\ref{ss:2.4}.
\begin{proposition}
\label{prop:regularized}
 Assume \eqref{e:hyp1}, \eqref{hyp:5},
and~\eqref{e:hyp2-bis-visco}. Assume that, in addition, $\phi$ satisfies
\begin{equation}
\label{e:addphi} \widehat{\phi} \in \mathrm{C}^2 (\R) \ \
\text{and} \ \ \exists\, C_{\phi,3}>0\,: \ \ \forall\, r \in \R
\quad |\phi{'}(r)| \leq C_{\phi,3}(1+|r|^4)\,.
\end{equation}
Then, for all $0 <\tau <T$,
 the pair $(\chi,w)$ has the further regularity
\begin{align}
\label{e:further-reg-chi} & \chi \in L^\infty(\tau, T; Z) \cap H^1
(\tau,T; V)\,,
\\
& \label{e:further-reg-w}
 \alpha(w) \in L^{\rho_p} (\tau, T; H^3(\Omega))\,.
\end{align}
In particular, if $\chi_0\in Z$, then the above properties hold for
any $\tau\in [0,T)$.
\end{proposition}
\begin{remark}
\label{uniform} \upshape
 From the proof of
Proposition~\ref{prop:regularized}, it is not difficult to recover a
uniform estimate of the following form:
\begin{equation}
\label{uniform-gronwall} \Vert\chi\Vert_{ L^\infty(\tau, T; Z) \cap
H^1 (\tau,T; V)} +
 \Vert\alpha(w)\Vert_{L^{\rho_p} (\tau, T; H^3(\Omega))}
 \le Q(\tau^{-1},\Vert \chi_0\Vert_V),
\end{equation}
where $Q$ is a suitable function which is nondecreasing with respect to both
arguments.
\end{remark}
%
\subsection{Well-posedness for the viscous problem}
\paragraph{Continuous dependence on the initial data and uniqueness.}
We will prove  uniqueness (and continuous dependence) results for
Problem~\ref{p:2} under more  restrictive assumptions on $\alpha$
and on the growth of the  function $\phi$.  In particular, we are
going to consider two sets of assumptions.

First, we will suppose that
 $\phi$ behaves like a polynomial of degree at most
$3$. For the sake of simplicity and without loss of generality, we
will carry out our analysis in the case when $\phi$ is the
derivative of the  double-well potential
$\widehat{\phi}(r)=(r^2-1)^2/4$. Furthermore, we will
replace~\eqref{e:hyp1} by
\begin{equation}
\label{e:hyp1-bis} \tag{H7}
\begin{gathered}
\alpha: \R \to \R \qquad \text{is a strictly increasing and
differentiable function such that}
\\
 \ \exists\, C_{9},\, C_{10} >0 \, : \quad \forall\, r
\in \R \qquad C_{9} \leq \alpha'(r) \leq C_{10}\,,
\end{gathered}
\end{equation}
and~\eqref{hyp:5}--\eqref{e:hyp2-bis-visco} by
\begin{equation}
\label{e:hyp7-bis} \tag{H8} \phi(r) =r^3-r \qquad \forall\, r \in
\R\,.
\end{equation}
\begin{theorem}
\label{th:3} Assume~\eqref{e:hyp1-bis} and~\eqref{e:hyp7-bis}.  Let
$\chi_{0}^{1}$ and $\chi_{0}^{2} $ be two
    initial  data
    for Problem~\ref{p:2}
    fulfilling~\eqref{hyp:initial-datum} and set
$M_{*}:= \max_{i=1,2} \{\| \chi_{0}^{i}\va \}
    $; let $ \chi_{i} $, $ i=1,2 $, be the corresponding
solutions.
    Then, for every $\delta>0$, there exists a positive constant $ S_{\delta} $, also depending on
\begin{equation}
\label{e:only-depe}
    \text{$ M_{*} $,
    $ T $, $ |\Omega| $, $C_{9}$, and $C_{10} $,}
\end{equation}
    such that
      \begin{equation}
      \label{contdepV}
\| \chi_1 (t) - \chi_2 (t) \|_{V}   + \| \chi_1 - \chi_2 \|_{H^1
(0,t;H) \cap L^2 (0,t;Z)}
 \leq
    S_{\delta}   \|\chi_{0}^{1} - \chi_{0}^{2} \|_{V}
           \quad \forall t \in
         [0,T].
    \end{equation}
\end{theorem}

Our second continuous dependence results holds in the more general
frame of assumptions of Proposition~\ref{prop:regularized}, but for
more regular initial data. Indeed, we have
\begin{theorem}
\label{th:3.2} Assume that \eqref{e:hyp1} holds for some $p\in
[0,1]$, and that $\phi$ complies with \eqref{hyp:5},
\eqref{e:hyp2-bis-visco}, and~\eqref{e:addphi}.  Let $\chi_{0}^{1}$
and $\chi_{0}^{2} $ be two
    initial  data
    for Problem~\ref{p:2} such that $\chi_{0}^{i} \in Z$ and $\widehat{\phi}(\chi_{0}^{i}) \in L^1
    (\Omega)$ for $i=1,2$, and  let $ \chi_{i} $, $ i=1,2 $, be the corresponding
solutions. Then, for every $\delta>0$, there exists a positive
constant $ S_{\delta} $, also depending on $T$, $|\Omega|$, $C_1$,
$C_2$ and $M^{*}:= \max_{i=1,2} \{\| \chi_{0}^{i}\|_Z \}$,  such
that
 estimate \eqref{contdepV}
holds for all $t \in [0,T]$.
\end{theorem}
\subsection{Global attractor and exponential attractors for the viscous problem}
\label{ss:2.4}
The \emph{energy functional}  associated with Problem~\ref{p:2}
reads
\begin{equation}
\label{e:ene-funct} \ene: \cx \to \R, \ \  \ \ \ene(v):= \frac12
\int_{\Omega} |\nabla v|^2 + \int_{\Omega} \widehat{\phi}(v) \ \
\text{for all $v \in \cx$.}
\end{equation}
Consequently, we introduce the phase space $(\cx,\dcx)$ of energy
bounded solutions, defined by
\begin{equation}
\label{e:pspace}
\begin{aligned}
& \cx= \left\{v \in V\, : \ \widehat{\phi}(v) \in L^1 (\Omega)
\right\},
\\
&  \dcx(v_1,v_2) = \| v_1 -v_2 \|_{H^1 (\Omega)} +
 \left\|\widehat{\phi}(v_1) -
{\widehat{\phi}(v_2)} \right\|_{L^1(\Omega)} \qquad \text{for all
$v_1, \, v_2 \in \cx$.}
\end{aligned}
\end{equation}

The following definition  details the properties of  the solutions to
Problem~\ref{p:2} to which our long-time analysis will apply.
\begin{definition}
\label{def:solp2} We say that a function $\chi : [0,+\infty) \to
\cx$ is a \emph{solution to Problem~\ref{p:2} on $(0,+\infty)$} if,
for all $T>0$, $\chi$ enjoys  regularity~\eqref{reg-chi-bis} on the
interval $(0,T)$  and there exists  a function $w$, with
regularity~\eqref{reg-w} for all $T>0$, such that
equations~\eqref{1-var-better}--\eqref{2-var-better} hold almost
everywhere on $\Omega \times (0,+\infty)$. We set
\begin{equation}
\label{e:gen-semiflow} \mathcal{S}= \left \{ \chi : [0,+\infty ) \to
X \, : \ \text{$\chi$ is a solution to Problem~\ref{p:2} on
$(0,+\infty)$} \right\}\,. \end{equation}
\end{definition}
\noindent

 We  assume that, besides~\eqref{e:hyp1}, $\alpha$
 complies with the following condition, slightly stronger
 than~\eqref{e:hyp1}:
 \begin{equation}
\label{e:hyp-add-alpha} \tag{H9}
 \exists\, \mathsf{c}_{\alpha}>0, \ \
\exists\, \Psi: \R \to [0,+\infty) \ \text{convex}\,:  \ \ \forall
\, r \in \R\, \quad \alpha(r) r - \mathsf{c}_{\alpha}|r|^{2p+2} =
\Psi(r)\,.
 \end{equation}
Hence,  our first result asserts that the solution set $\mathcal{S}$
is a \emph{generalized semiflow} in the sense of
Definition~\ref{def:generalized-semiflow}.
\begin{proposition}
\label{prop:2.1} Assume \eqref{e:hyp1},
\eqref{hyp:5}--\eqref{e:hyp2-bis-visco}.
 Then,
\begin{enumerate}
\item every $\chi  \in \mathcal{S}$ (cf.~\eqref{e:gen-semiflow})  complies with the \emph{energy identity}
\begin{equation}
\label{e:enid} \delta\int_s^t \int_{\Omega} |\chi_t|^2 + \int_s^t
\int_{\Omega} \alpha'(w)|\nabla w|^2 + \ene(\chi(t)) = \ene(\chi(s))
\quad \text{for all $0 \leq s \leq t$,}
\end{equation}
the function $w:(0,+\infty) \to V$ being defined
by~\eqref{2-var-better} on $\Omega \times (0,+\infty)$.
\item Assume that $\alpha$ in addition complies with~\eqref{e:hyp-add-alpha}.
Then, the set $\mathcal{S}$ is a generalized semiflow in the
 phase space~\eqref{e:pspace}, and its elements are continuous
 functions from $[0,+\infty)$ onto $X$.
 \end{enumerate}
\end{proposition}
\noindent
 We prove our main result on the long-time behavior of the
solutions to Problem~\ref{p:2}  under a further condition on $\phi$,
which in particular implies (and thus replaces)
\eqref{e:hyp2-bis-visco}, namely
\begin{equation}
\label{lim-infty-phi} \tag{H10}
\begin{aligned}
\lim_{r\to +\infty}\phi(r)=+\infty, \qquad \lim_{r\to
-\infty}\phi(r)=-\infty\,,
\\
\lim_{r\to +\infty}\phi'(r)=\lim_{r\to -\infty}\phi'(r)=+\infty\,.
\end{aligned}
\end{equation}
\begin{maintheorem}
\label{th:4} Assume \eqref{e:hyp1}, \eqref{hyp:5},
\eqref{e:hyp-add-alpha}, and \eqref{lim-infty-phi}.
  For a given $\mathrm{m}_0
>0$, denote by $\mathcal{D}_{\mathrm{m}_0}$ the set
\begin{equation}
\label{e:fixed-mean-value} \mathcal{D}_{\mathrm{m}_0}= \left \{ \chi
\in X\, : \ |m(\chi)| \leq {\mathrm{m}_0} \right\}\,.
\end{equation}
Then, the semiflow $\mathcal{S}$ possesses a unique global attractor
$\mathcal{A}$ in $\mathcal{D}_{\mathrm{m}_0}$, given by
\begin{equation}
\label{e:att} \mathcal{A}: =\bigcup\left \{\omega(D)\, : \ D \subset
\mathcal{D}_{\mathrm{m}_0} \ \text{bounded}\right\}\,.
\end{equation}
Finally, we have the following enhanced regularity for the elements
of the $\omega$-limit of every trajectory:
 \begin{equation}
\label{e:enhanced-regularity} \forall\, p \in [1,+\infty) \ \
\exists\,C_p>0\,: \ \ \forall\,\chi \in \mathcal{S}, \ \ \forall\,
\bar{\chi} \in \omega(\chi) \quad \| \bar{\chi}\|_{W^{2,p}(\Omega)}
+ \| \widehat{\phi}(\bar{\chi})\|_{L^p(\Omega)} \leq C_p\,.
 \end{equation}
\end{maintheorem}
\begin{remark}
\label{rem:poly-advantages} \upshape
 Notice that, in the  case
 \begin{equation}
 \label{e:special-case}
 \text{
 $\widehat{\phi}$ is a polynomial  of even degree $\mathsf{m} \geq 4$, with a positive leading
 coefficient,}
 \end{equation}
then  conditions \eqref{hyp:5} and \eqref{lim-infty-phi}
 are  satisfied.
\end{remark}
\begin{remark}[Enhanced regularity of the global attractor]
\label{zeta} \upshape In addition to hypotheses~\eqref{e:hyp1},
\eqref{hyp:5}, \eqref{e:hyp-add-alpha},   and~\eqref{lim-infty-phi}
of Theorem~\ref{th:4}, assume that $\phi$ complies
with~\eqref{e:addphi}. Then, the enhanced regularity
estimate~\eqref{e:further-reg-chi} holds for $\chi$.

 This
regularity is reflected in the further regularity
\begin{equation}
\label{e:further-reg-att}
 \mathcal{A}\subset Z, \end{equation} for the global
attractor $\mathcal{A}$, which holds  provided that one works with
the (smaller) set of solutions to Problem~\ref{p:2} arising from the
approximation procedure which will be detailed in  Appendix. In
fact, the estimates leading to~\eqref{e:further-reg-chi} can be
rigorously justified only  for this approximate problem, as we will
see in the proof of Proposition~\ref{prop:regularized}, cf.
Section~\ref{s:3.3}. Now, the aforementioned family of
``approximable'' solutions to Problem~\ref{p:2} (see, e.g.,
\cite{BabinVishik92, rossi-segatti-stefanelli08, schimperna07,
segatti06}  for analogous constructions) complies with the
properties defining a generalized semiflow, except for the
concatenation axiom. This has motivated the introduction
in~\cite{rossi-segatti-stefanelli08, segatti06} of the (slightly
more general) notion of  \emph{weak} global attractor, tailored to
the \emph{weak} generalized semiflows without the concatenation
property. Hence, relying on the abstract results
of~\cite{rossi-segatti-stefanelli08, segatti06} and  arguing as in
the proof of Theorem~\ref{th:4}, it is straightforward to prove that
the semiflow associated with the approximable solutions to
Problem~\ref{p:2} admits a \emph{weak} global attractor for
which~\eqref{e:further-reg-att} holds. On the other hand,
 Theorem~\ref{th:5} below shows that,
under the stronger assumptions of Theorem~\ref{th:3.2},
the semiflow possesses the standard global attractor $\mathcal{A}$
satisfying \eqref{e:further-reg-att}, namely, $\mathcal{A}$ is a
compact and invariant set which attracts (in the $V$-metric) all
bounded sets of initial data as time goes to infinity.
\end{remark}

We conclude this section by showing that it is also possible to
construct an exponential attractor through the short-trajectories
approach developed in \cite{malek-prazak}. Let us first set
 for a given $\tau>0$
\begin{equation*}
X_\tau=L^2(0,\tau;V),\quad Y_\tau=\left\{u\in L^2(0,\tau;Z)\,:\, u_t \in L^2(0,\tau;H)\right\}
\end{equation*}
and observe that $Y_\tau$ is compactly embedded in $X_\tau$.

Under assumptions \eqref{e:hyp1}, \eqref{hyp:5},
\eqref{e:hyp2-bis-visco}, and~\eqref{e:addphi}, we know that, for
any $\chi_0\in V$ and any $T>0$, there exists a pair $(\chi,w)$
which solves Problem~\ref{p:2}  with the regularity
 \eqref{reg-chi-bis}, \eqref{reg-w},
\eqref{e:further-reg-chi}, \eqref{e:further-reg-w}
(cf.~Theorem~\ref{th:2} and Proposition~\ref{prop:regularized}). In
particular, $\chi\in Y_T$. In addition,  thanks to~\eqref{e:enid}
and arguing in the same way as in the forthcoming
Section~\ref{s:3.1},  it is not difficult to show that $\Vert
\chi\Vert_{Y_T}$ can be estimated uniformly with respect to
$\Vert\chi_0\Vert_V$, . The energy identity also entails the
existence of a bounded set   $B^0\subset V$  such that, if
$(\chi,w)$ is a solution to Problem~\ref{p:2} with the
aforementioned properties, then there exists $t_0>0$, only depending
on $\Vert\chi_0\Vert_V$, such that $\chi(t)\in B^0$ for all $t\geq
t_0$ and $\chi(t)\in B^0$ for all $t\geq 0$ whenever $\chi_0\in B^0$
(see the proof of the eventual boundedness of $\mathcal{S}$ in
Section~\ref{ss:5.2}). Let us now consider the set
$\mathcal{X}_\ell=\{\chi : (0,\ell)\to V\}$ of all the
$\ell$-trajectories $\chi$ such that $(\chi,w)$ is a solution to
Problem~\ref{p:2} satisfying \eqref{reg-chi-bis}, \eqref{reg-w},
\eqref{e:further-reg-chi}, \eqref{e:further-reg-w}. Then, we endow
this set with the $X_\ell$-topology (note that it might be a
non-complete metric space). Moreover, denoting by $V_w$ the space
$V$ endowed with the weak topology, we have $\mathcal{X}_\ell\subset
C^0([0,\ell];V_w)$. Consequently, any $\ell$-trajectory makes sense
pointwise.

From now on, we assume that assumption \eqref{e:hyp1} holds for some
$p\in[0,1]$. Thanks to \eqref{e:further-reg-chi}, for any
$\ell$-trajectory, there exists $\tau\in (0,\ell)$ such that
$\chi(\tau)\in Z$.  This is sufficient to  conclude that $\chi$ is
unique from $\tau$ on, as   a consequence of
Proposition~\ref{prop:regularized} and Theorem~\ref{th:3.2}.
Therefore,  if $\chi\in \mathcal{X}_\ell$ and $T>\ell$, then there
exists a unique $\tilde \chi\in\mathcal{X}_T$ such that
$\tilde\chi\vert_{[0,\ell]} = \chi$. Thus, we can define a semigroup
$L_t$ on $\mathcal{X}_\ell$ by setting
\begin{equation*}
(L_t\chi)(\tau):= \tilde\chi(t+\tau),\qquad \tau\in [0,\ell],
\end{equation*}
where $\tilde\chi$ is the unique element of $\mathcal{X}_{\ell+\tau}$ such that
$\tilde\chi\vert_{[0,\ell]} = \chi$.

Let us now set
\begin{equation*}
B^0_\ell:=\left\{\chi\in\mathcal{X}_\ell\,:\, \chi(0)\in B^0\right\}.
\end{equation*}
Then, by Proposition~\ref{prop:regularized}, we can infer that the
set $\left\{\chi\vert_{[\ell/2,\ell]}\,:\, \chi\in B^0_\ell\right\}$
is bounded in $L^\infty(\ell/2,\ell;Z)$. Hence, we can prove
 a continuous dependence  estimate  like \eqref{contdepV}, which allows us to apply
\cite[Lemma~2.1]{malek-prazak} and deduce that $L_t$ is Lipschitz
continuous on $B^0_\ell$, uniformly with respect to $t\in [0,\tau]$
for any fixed $\tau>0$. Observe that, arguing as in
Section~\ref{s:3.3}, we can prove that
$B^1_\ell=\overline{L_\tau(B^0_\ell)}^{X_\ell} \subseteq B^0_\ell$
for some $\tau>0$.  From this fact we  deduce  that the dynamical
system $(\mathcal{X}_\ell,L_t)$ has a global attractor
$\mathcal{A}_\ell$ (see \cite[Thm.~2.1]{malek-prazak}). In addition,
$L_\tau: \mathcal{X}_\ell \to Y_\ell$ is Lipschitz continuous for
some $\tau>0$. Indeed, recall that $B^1_\ell$ is bounded in
$L^\infty(0,\ell;Z)\cap H^1(0,\ell;V)$ and use \eqref{contdepV}.
Thus, on account of \cite[Thm.~2.2]{malek-prazak}, we can infer that
$\mathcal{A}_\ell$ has finite fractal dimension. In order to go back
to the original geometric space $V$, we introduce the evaluation
  mapping $e: \mathcal{X}_\ell \to V,$
$e(\chi):=\chi(\ell)$. Then, we set $B^1:=e(B^1_\ell)$ and we note
that, for any $\chi_0\in B^1$, there is a unique solution to
Problem~\ref{p:2}, so that the solution operator $S_t$ is well
defined on $B^1$ and $S_t(B^1)\subseteq B^1$, for all $t\geq 0$. In
addition, $e$ is (Lipschitz) continuous on $B^1_\ell$ (use
\eqref{contdepV} and \cite[Lemma~2.1]{malek-prazak} once more).
Therefore, we use \cite[Thm.~2.4]{malek-prazak} to deduce that
$\mathcal{A}:=e(\mathcal{A}_\ell)$ is the finite-dimensional global
attractor of the dynamical system $(B^1,S_t)$.

 It remains  to prove the existence of an exponential attractor. We already
know that $L_t$ is Lipschitz continuous on $B^1_\ell$, uniformly
with respect to $t\in [0,\tau]$ for every fixed $\tau>0$ (see
above). Thus, we only need to show that $t\mapsto L_t\chi$ is
H\"older continuous with values in $V$, uniformly with respect to
$\chi\in B^1_\ell$. This follows from
\cite[Lemma~2.2]{malek-prazak}, recalling that $B^1_\ell$ is, in
particular, bounded in $H^1(0,\ell;V)$. Hence,
$(\mathcal{X}_\ell,L_t)$ has an exponential attractor
$\mathcal{E}_\ell$ and $\mathcal{E}:=e(\mathcal{E}_\ell)$ is an
exponential attractor for $(B^1,S_t)$.

Summing up, we have proved the

\begin{maintheorem}
\label{th:5} Assume that \eqref{e:hyp1} holds for some $p\in[0,1]$. Also,
assume \eqref{hyp:5}, \eqref{e:hyp2-bis-visco}, and~\eqref{e:addphi}. Then, there exists a
bounded invariant set $B^1\subset V$ such that Problem~\ref{p:2}
generates a dynamical system $(B^1,S_t)$ which possesses an
exponential attractor $\mathcal{E}$. In addition, the system also has
a global attractor $\mathcal{A}$ with finite fractal dimension.
\end{maintheorem}
\noindent
 Note that, in the framework of Theorem~\ref{th:5}, neither
 assumption~\eqref{e:hyp-add-alpha} nor \eqref{lim-infty-phi} are
 needed.
\section{Proofs of Theorems~\ref{th:1} and~\ref{th:2}}
\label{s:3}
\paragraph{Scheme of the  proofs of Theorems~\ref{th:1} and~\ref{th:2}.}
We will prove  Theorems~\ref{th:1} and~\ref{th:2} by taking the
limit of a suitable approximation scheme for Problems~\ref{p:1} and
\ref{p:2}. For the sake of readability, we postpone detailing
 such a scheme
in Appendix.

In Section~\ref{s:3.1}, we will instead   perform  all  estimates
leading to the aforementioned passage to the limit directly on
systems~\eqref{1-var}--\eqref{2-var}
and~\eqref{1-var-better}--\eqref{2-var-better}. Note that, at this
stage, some of the following calculations will only be formal,
cf.~Remark~\ref{solo-formale} below. Their rigorous justification
will be given in  Appendix, see Section~\ref{ss:a.1}.

Next, in Section~\ref{s:3.2} (in Section~\ref{s:3.3}, respectively), we
will carry out a passage to the limit in some unspecified
approximation scheme for Problem~\ref{p:1} (for Problem~\ref{p:2},
respectively) and conclude the (formal) proof of Theorem~\ref{th:1} (of
Theorem~\ref{th:2}, respectively).  In Section~\ref{ss:a.2}, we will adapt the
limiting arguments developed in Sections~\ref{s:3.2} and~\ref{s:3.3}
to  the approximation scheme for Problems~\ref{p:1} and~\ref{p:2}
and carry out the rigorous proofs of the related existence
theorems.
\begin{notation}
\label{not:3.1} \upshape
 We will  perform the
  a priori estimates on
 systems~\eqref{1-var}--\eqref{2-var} and~\eqref{1-var-better}--\eqref{2-var-better},
 distinguishing the ones which hold both in the viscous and  the
 non-viscous cases from the ones which depend
 on the constant
 $\delta$ in~\eqref{2-var-better} (which can be either strictly
  positive or
equal to zero), and
 on our different assumptions on the nonlinearity $\phi$ in the viscous and non-viscous cases.
 Accordingly, we will use the generic notation $C$  for most of the
 constants
appearing in the forthcoming calculations and
  depending on the problem
 data, and $C_\delta$ ($C_0$, respectively) for those constants \emph{substantially}
  depending on the problem data and on
 $\delta >0$
 (on  $\delta =0$, respectively).
  We will adopt the same
 convention for the constants $S^i$, $S^i_\delta$, $S_0^i$, $i\geq 1$.
\end{notation}
\subsection{A priori estimates}
\label{s:3.1}
\paragraph{First a priori estimate.} We
test~\eqref{1-var-better} by $w$, \eqref{2-var-better} by $\chi_t$,
add the resulting relations, and integrate over some time interval
$(0,t) \subset (0,T)$. Elementary calculations lead to
\begin{equation}
\label{est:1} \int_0^t \int_{\Omega} \alpha'(w) |\nabla w|^2 +
\delta \int_0^t \| \chi_t \|_{H}^2 + \frac12 \| \nabla \chi(t)
\|_{H}^2 + \int_{\Omega} \widehat{\phi} \left( \chi(t)\right)
=\frac12 \| \nabla \chi_0 \|_{H}^2 + \int_{\Omega} \widehat{\phi}
\left( \chi_0\right)\,.
\end{equation}
Recalling~\eqref{hyp:initial-datum}, the second of~\eqref{e:hyp1}
(which, in particular, yields that $\alpha'$ is bounded from below on
$\R$ by a positive constant)  and the positivity of
$\widehat{\phi}$ (cf.~\eqref{phicappucciopos}), we conclude that, for
some constant $S^1
>0$,
\begin{equation}
\label{e:stima1} \| \nabla w \|_{L^2 (0,T;H)} + \| \nabla \chi
\|_{L^\infty (0,T;H)} + \| \widehat{\phi}(\chi)\|_{L^\infty (0,T;L^1
(\Omega))} \leq S^1.
\end{equation}
\paragraph{First a priori estimate in the viscous case.} In the case
$\delta>0$,  from the previous a priori estimate we also have
\begin{equation}
\label{e:stima1-delta}  \| \chi_t \|_{L^2 (0,T;H)} + \|
\chi\|_{L^\infty (0,T;V)} \leq S_\delta^1\,.
\end{equation}
\paragraph{Second a priori estimate.} We test~\eqref{1-var-better}
by $1$ and find $\mo(\chi_t)=0$ a.e. in $(0,T)$, so that, in
particular,
\begin{equation}
\label{e:constant-m-value}
 \mo(\chi(t))=m_0:=\mo(\chi_0) \qquad \forall\, t \in [0,T].
\end{equation}
 Hence,
testing~\eqref{2-var-better} by $1$, we obtain
\begin{equation}
\label{est:2} \mo(\phi(\chi(t))) = \mo(w(t)) \qquad \forae \  t \in
(0,T).
\end{equation}
\paragraph{Second a priori estimate in the non-viscous case.} It
follows from \eqref{e:constant-m-value} and the Poincar\'e
inequality that
\begin{equation}
\label{est:2-deltazero} \| \chi\|_{L^\infty (0,T;V)} \leq S^2\,.
\end{equation}
\paragraph{Third a priori estimate in the non-viscous case.}
We test~\eqref{2-var} by $\chi-\mo(\chi)$: we have, for a.e. $t \in
(0,T)$,
\begin{equation}
\label{est:3}
\begin{aligned}
 \| \nabla \chi(t)  \|_{H}^2 +
\int_{\Omega} \phi(\chi(t)) \left( \chi(t) - \mo(\chi(t)) \right) &=
\int_{\Omega} w(t) \left( \chi(t) - \mo(\chi(t)) \right)  \\ & =
\int_{\Omega} \left( w(t) - \mo(w(t)) \right) \left( \chi(t) -
\mo(\chi(t)) \right)\\ &  \leq  C \| \nabla \chi(t) \|_{H} \| \nabla
w(t) \|_{H} \leq C S^1 \| \nabla w(t) \|_{H} ,
\end{aligned}
\end{equation}
the latter estimate ensuing from the Poincar\'e inequality for zero
mean value functions and the pre\-vious~\eqref{e:stima1}. On the
other hand, \eqref{e:2.2.14} and~\eqref{e:constant-m-value}  yield
that there exist  constants $C_{m_0}, C_{m_0}'>0$ such that, for
a.e.  $t \in [0,T]$,
\begin{equation}
\label{e:useful-later} \int_{\Omega} |\phi( \chi(t))| \leq C_{m_0}
\int_{\Omega} \phi(\chi(t)) \left( \chi(t) - \mo(\chi(t)) \right) +
C_{m_0}' \,.
\end{equation}
Combining this with \eqref{est:3}, we deduce that there exists $C>0$,
also depending on $C_{m_0}$ and on $C_{m_0}'$, such that,  for a.e.
$t \in (0,T)$,
\begin{equation}
\label{est:4}
 \int_{\Omega} |\phi( \chi(t))| \leq C \left(\| \nabla
w(t) \|_{H} + 1\right)\,.
\end{equation}
 Thus, in view
of~\eqref{e:stima1}, we obtain an estimate for $\phi(\chi)$ in $L^2
(0,T; L^1 (\Omega))$. Finally, due to~\eqref{est:2}, we find
$$
\| \mo(w) \|_{L^2 (0,T)} \leq C_0.
$$
Hence, by~\eqref{e:stima1} and the Poincar\'e inequality, we conclude
that
\begin{equation}
\label{e:stima3} \| w \|_{L^2 (0,T;V)} \leq S^1_0\,.
\end{equation}
\paragraph{Third a priori estimate in the viscous case.}
Estimate~\eqref{e:stima1} for $\widehat{\phi}(\chi)$
and~\eqref{hyp:5} yield that
\begin{equation}
\label{est:3-delta} \| \phi(\chi) \|_{L^\infty (0,T; L^1 (\Omega))
 } \leq S_\delta^2\,.
\end{equation}
Recalling~\eqref{est:2}, we immediately infer that
\begin{equation}
\label{est:4-delta} \| \mo(w) \|_{L^\infty (0,T)} \leq S_\delta^3,
\end{equation}
whence, again,
\begin{equation}
\label{est:5-delta} \| w \|_{L^2 (0,T; V)} \leq S_\delta^4.
\end{equation}
\paragraph{Fourth a priori estimate in the non-viscous case.}
We preliminarily observe that, thanks to~\eqref{e:hyp2-aftermath},
equation~\eqref{2-var} can be rewritten as
\begin{equation}
\label{e:more-convenient_form} A\chi + \beta(\chi) = w + C_{\phi,1}
\chi \qquad \aein \ \Omega \times (0,T)\,.
\end{equation}
 Notice that,
in view of~\eqref{est:2-deltazero} and \eqref{e:stima3}, the
right-hand side of~\eqref{e:more-convenient_form} belongs to
$L^2 (0,T; L^6 (\Omega))$. Hence, we can
test~\eqref{e:more-convenient_form} by
 $|\beta(\chi)|^4 \beta(\chi)$ and easily conclude  that
\[
 \|A\chi\|_{L^2 (0,T; L^6 (\Omega))} +\|\beta(\chi)\|_{L^2 (0,T;
L^6 (\Omega))} \leq C_0\,.
\]
Then,
 also  by standard elliptic regularity results, we find
\begin{equation}
\label{e:stima3-bis} \| \phi(\chi)\|_{L^2 (0,T; L^6 (\Omega))} +  \|
\chi \|_{L^2 (0,T;W^{2,6} (\Omega))} \leq S^2_0\,.
\end{equation}
\paragraph{Fourth a priori estimate in the viscous case.}
We combine~\eqref{e:stima1-delta}  and \eqref{est:5-delta} and
argue by comparison in~\eqref{2-var-better}. Relying
on~\eqref{e:hyp2-bis-visco} and on the related elliptic regularity
estimate, we have
\begin{equation}
\label{e:useful-again} \| \phi(\chi)  \|_{L^2 (0,T;H)} \leq
S^5_\delta\,,
\end{equation}
 as well as an estimate for $A\chi$  in
$L^2 (0,T; H)$, so that
\begin{equation}
\label{e:stima3-delta} \| \chi \|_{L^2 (0,T;Z)} \leq S^6_\delta\,.
\end{equation}
\paragraph{Fifth a priori estimate.} It follows from~\eqref{est:1}
and~\eqref{e:hyp1} that $\int_0^T \int_{\Omega} w^{2p} |\nabla w|^2
\leq C$, whence we conclude that
\begin{equation}
\label{est:3.14} \| \nabla (|w|^{p}w) \|_{L^2 (0,T;H)} \leq S^3\,.
\end{equation}
\paragraph{Sixth a priori estimate in the non-viscous case.}
From~\eqref{e:stima1}, \eqref{hyp:3}, and \eqref{e:stima3-bis},   we
deduce that
\begin{equation}
\label{e:crucial0}
 \||\phi(\chi)|^{\sigma}\|_{L^{2/\sigma}(0,T; L^{6/\sigma} (\Omega)) \cap L^\infty (0,T; L^1
 (\Omega))} \leq C_0\,.
\end{equation} Using the interpolation inequality
$$
\forall\, v \in L^1 (\Omega) \cap L^{6/\sigma} (\Omega) \quad  \| v
\|_{L^{1/\sigma} (\Omega)} \leq \| v \|_{L^1 (\Omega)}^\teta \, \| v
\|_{L^{6/\sigma} (\Omega)}^{1-\teta}, \quad \text{with $\teta =
\frac{5\sigma}{6-\sigma}$,}
$$
we obtain the estimate
\begin{equation}
\label{e:crucial}
 \||\phi(\chi)|^{\sigma}\|_{L^{q_\sigma}(0,T; L^{1/\sigma}
 (\Omega))} \leq C_0, \quad
  \text{with $q_\sigma=\frac{2}{\sigma} \, \frac{1}{1-\teta}= \frac{6-\sigma}{3\sigma-3\sigma^2}$,}
\end{equation}
whence a bound for $\phi(\chi)$ in $L^{\sigma q_\sigma} (0,T; L^1
(\Omega))$. Taking into account~\eqref{est:2}, we  conclude that
\begin{equation}
\label{e:3.15} \| \mo(w) \|_{L^{\sigma q_\sigma}(0,T)} \leq C_0\,,
\quad \text{whence} \quad \| |\mo(w)|^{p+1} \|_{L^{(\sigma
q_\sigma)/(p+1)}(0,T)} \leq C_0\,.
\end{equation}
On the other hand,
 applying the \emph{nonlinear} Poincar\'e
inequality~\eqref{e:poinc-gen} with the choices $X=V$, $Y=H$,
 $Gv= \nabla v$, and $\Psi(v)= |\Omega|^{-p-1}|\int_{\Omega} |v|^{\frac{1}{p+1}} \mathrm{sign}(v)|^{p+1}$, where
 $v=|w|^p w$, we find
 \begin{equation}
\label{e:poinc-est} \||w|^p w  \|_{V} \leq K \left( \| \nabla
(|w|^{p}w) \|_{H} + \left| \mo(w)\right|^{p+1}\right)\,.
 \end{equation}
Therefore,  combining  estimate~\eqref{e:3.15} for  $|\mo(w)|^{p+1}$
with \eqref{est:3.14}, we finally obtain, owing to the Poincar\'e inequality
\eqref{e:poinc-est},
\begin{equation}
\label{e:to-be-cited3} \| |w|^p w\|_{L^{(\sigma
q_\sigma)/(p+1)}(0,T; V)} \leq C_0.
\end{equation}
 Using the embedding $V \subset L^6
(\Omega)$ and the growth~\eqref{e:conse1} for $\alpha$, we
 infer
\begin{subequations}
\label{e:est-refined}
\begin{equation}
\label{e:est-refined_1}
 \| \alpha(w) \|_{L^{\eta_{p\sigma}} (0,T; L^{\kappa_p} (\Omega))}
 \leq S_0^3\,
\end{equation}
(where we have used the fact that $(\sigma q_\sigma)/(2p+1)$ equals
the index $\eta_{p\sigma}$ defined in~\eqref{e:index-notation}).
Hence,  by comparison in~\eqref{1-var}, we also conclude that
\begin{equation}
\label{e:3.18} \| \chi_t \|_{L^{\eta_{p\sigma}}} (0,T;
\sobneg{-2}{{\kappa_p}}) \leq S_0^4
\end{equation}
\end{subequations}
(see again~\eqref{e:index-notation} for the definition of
$\kappa_p$).
\paragraph{Sixth a priori estimate in the viscous case.}
Combining~\eqref{est:4-delta}, \eqref{est:3.14} and the
Poincar\'e-type inequality~\eqref{e:poinc-est}, we deduce an
estimate for $|w|^{p}w $ in $L^2 (0,T;V)$. Then, arguing in the same
way as for~\eqref{e:est-refined_1}, we have
\begin{equation}
\label{e:to-be-quoted}
 \| \alpha(w) \|_{L^{\rho_p} (0,T; L^{\kappa_p}
(\Omega))} \leq C_\delta,
\end{equation}
 the index $\rho_p$ being defined
in~\eqref{e:index-notation}.
 Now, in view of estimate~\eqref{e:stima1-delta}
for $\chi_t$ in $L^2 (0,T;H)$, a comparison in~\eqref{1-var-better}
yields an estimate for $A(\alpha (w))$ in $L^2 (0,T;H)$. By elliptic
regularity results, we finally conclude that
\begin{equation}
\label{e:3.19}
 \| \alpha(w) \|_{L^{\rho_p} (0,T; Z)} \leq S_\delta^7\,.
\end{equation}
\paragraph{Seventh a priori estimate in the non-viscous case.}
Our aim is now to show that
\begin{equation}
\label{useful2}
 \| \phi(\chi) \|_{L^{\sigma q_\sigma} (0,T;L^{6}(\Omega))} \leq S_0^5, \qquad \text{with $\sigma q_\sigma =
\frac{6-\sigma}{3-3\sigma}>2$\,.}
\end{equation}
Indeed, again recalling the embedding $V \subset L^6 (\Omega)$, we
observe that \eqref{e:to-be-cited3} yields an estimate for $w$ in
$L^{\sigma q_\sigma}(0,T; L^{6p+6}(\Omega))$. Then, taking into
account estimate~\eqref{est:2-deltazero} for $\chi$ in $L^\infty
(0,T; L^6 (\Omega))$, together with   the aforementioned elliptic regularity
argument, we find estimate~\eqref{useful2} by a comparison
in~\eqref{e:more-convenient_form}.
\paragraph{Seventh a priori estimate in the viscous case.}
 We combine estimate~\eqref{e:3.19}, the continuous embedding $Z \subset
 L^\infty(\Omega)$, and the growth condition~\eqref{e:conse1} to
 deduce  an estimate for
 $w$ in $L^{\rho_p (2p+1)} (0,T;L^\infty (\Omega))$,
 whence
\begin{equation}
\label{e:3.22} \| w\|_{L^{2p+2} (0,T; L^{\infty}(\Omega))} \leq
S_\delta^8\,.
\end{equation}
\begin{remark}
\upshape \label{solo-formale} Notice that all the
 a priori estimates   for the viscous Problem~\ref{p:2}
     are in fact rigorously justified   on
 system~\eqref{1-var-better}--\eqref{2-var-better}. This  has
 significant repercussions on the long-time analysis of
 Problem~\ref{p:2}. Indeed, this allows us to work with the semiflow
 associated with the solutions to Problem~\ref{p:2} (cf.~\eqref{e:gen-semiflow})
 and prove the existence of a global attractor in the sense of \cite{Ball97}.
 However, as pointed out in Remark~\ref{zeta}, if we  address
  further regularity
 properties of the attractor (e.g., \eqref{e:further-reg-att}),
 then we need additional estimates which cannot be performed
 directly on
 system~\eqref{1-var-better}--\eqref{2-var-better}, due to insufficient regularity of the solutions. Thus, we
  have to rely  on some
 approximation. On the one hand, this leads to a smoother
 attractor $\mathcal{A}$, but, on the other hand, we lose the concatenation
 property of the trajectories (cf.~\cite{rossi-segatti-stefanelli08, segatti06});
 moreover,
 only  trajectories which are limits of the approximation
 scheme will be attracted by the smoother attractor $\mathcal{A}$.

 We also point out that the viscous system~\eqref{1-var-better}--\eqref{2-var-better}
 cannot be used as an approximation for the non-viscous problem.
Indeed, it is not difficult to realize that  the fourth a priori
estimates~\eqref{e:more-convenient_form}--\eqref{e:stima3-bis}
(yielding a bound for $\phi(\chi)$ which plays a crucial role in the
ensuing calculations)
 are not compatible with the  term
$\delta\chi_t$ in~\eqref{2-var-better}.

This fact seems to suggest  the use of two different approximation
schemes for Problem~\ref{p:1} and Problem~\ref{p:2}, which would
lead to cumbersome and repetitious calculations.
 In order to circumvent this problem, we will construct in
 Appendix an approximation scheme depending on two distinct
 parameters and prove the existence of solutions to
 Problem~\ref{p:1} by passing to the limit in  three
 steps. Since the (rigorous) proof of existence for Problem~\ref{p:2} can be
 performed along the very same lines, we have chosen not to detail
 it in  Appendix.
\end{remark}
\subsection{Proof of Theorem \ref{th:1}}
\label{s:3.2} Let $\{ (\chi_n, w_n )\} $ be some sequence of
approximate solutions to Problem~\ref{p:1}. Due to
estimates~\eqref{e:stima1}, \eqref{est:2-deltazero},
\eqref{e:stima3}, \eqref{e:stima3-bis}, and~\eqref{e:est-refined},
applying standard compactness and weak compactness results
(see~\cite{simon}),  we find that there exists a pair $(\chi,w)$
with the regularities specified by~\eqref{reg-chi}--\eqref{reg-w0}
such that, along a (not relabeled) subsequence, the following
strong, weak, and weak$^*$ convergences hold as $n \to +\infty$:
\begin{align}
 \label{e:conv-chi-strong}
 &
\begin{aligned}
  \chi_n \to \chi \quad &\text{in \ $L^2
(0,T; W^{2-\eps,6} (\Omega)) \cap L^q (0,T;V) \cap \CC^0 ([0,T];
H^{1-\eps} (\Omega))$} \\ &  \text{for  every $\eps >0 $ and $1 \leq
q < +\infty$,}
\end{aligned}
\\
& \label{e:conv-chi-weak1}
 \chi_n \weaksto \chi \quad  \text{in \ $L^2
(0,T; W^{2,6} (\Omega)) \cap L^\infty (0,T;V)$,} \\
& \label{e:conv-chi-weak2}
  \chi_{n,t}
\weakto \chi_t \quad \text{in \ $L^{\eta_{p\sigma}} (0,T;
\sobneg{-2}{\kappa_p})$,}
\\
& \label{e:conv-w-strong} w_n  \weakto w \quad \text{in \ $L^2 (0,T;
V)$.}
\end{align}
Furthermore, there exists $\bar{\alpha} \in L^{\eta_{p\sigma}} (0,T;
L^{\kappa_p} (\Omega))  $ such that
\begin{equation}
\label{e:conv-alpha-weak} \alpha(w_n) \weakto \bar{\alpha} \quad
\text{in \ $L^{\eta_{p\sigma}} (0,T; L^{\kappa_p} (\Omega))$.}
\end{equation}
Now, estimate \eqref{useful2} in particular yields (recall that
$\sigma q_\sigma >2$) that
\begin{equation}
\label{e:cruc-pass1} \text{the sequence $\{ \phi(\chi_n)\}$ is
uniformly integrable in $L^2 (0,T;H)$.}
\end{equation}
  Furthermore, we have, up  to a further
subsequence,
\begin{equation}
\label{e:cruc-pass2} \phi(\chi_n(x,t)) \to \phi(\chi(x,t)) \qquad
\forae\, (x,t) \in \Omega \times (0,T)\,,
\end{equation}
which is a consequence of the continuity of $\phi$ and of the
pointwise convergence (up to a further subsequence)
 $\chi_n(x,t) \to \chi(x,t)$ a.e. in $\Omega \times (0,T)$ (cf.~\eqref{e:conv-chi-strong}).
Combining~\eqref{e:cruc-pass1} and~\eqref{e:cruc-pass2} and
 recalling the compactness criterion Theorem~\ref{t:ds}, we
conclude that
\begin{equation}
\label{e:strong-phi} \phi(\chi_n) \to \phi(\chi) \qquad \text{in \
$L^2 (0,T; H). $}
\end{equation}
Exploiting \eqref{e:conv-chi-strong}--\eqref{e:strong-phi}, one
easily concludes that the triplet $(\chi,w,\bar{\alpha})$ satisfies
\begin{align}
& \label{1-var-lim} \chi_t + A \bar{\alpha} =0 \qquad \text{in
$\sobneg{-2}{\kappa_p}$} \quad \aein \ (0,T)\,,
\\
& \label{2-var-lim} A\chi + \phi(\chi) =w \qquad \aein \ \Omega
\times (0,T)\,.
\end{align}
Finally, in order to prove that
\begin{equation}
\label{final-step} \bar{\alpha} (x,t) = \alpha (w(x,t)) \qquad
\forae\ (x,t) \in \Omega \times (0,T)\,,
\end{equation}
we test the equation approximating \eqref{2-var} by $w_n$ and
integrate in time. We thus have
$$
\begin{aligned}
\lim_{n \to +\infty}\int_0^T \int_{\Omega} |w_n|^2 &= \lim_{n \to
+\infty} \int_0^T \int_{\Omega}\phi(\chi_n) w_n + \lim_{n \to +\infty}
\int_0^T \int_{\Omega}\nabla \chi_n \cdot \nabla w_n \\ & = \int_0^T
\int_{\Omega}\phi(\chi) w+ \int_0^T \int_{\Omega}\nabla \chi
\cdot\nabla w\\ & =\int_0^T \int_{\Omega} |w|^2,
\end{aligned}
$$
where the second equality follows from
convergences~\eqref{e:conv-chi-strong}, \eqref{e:conv-w-strong},
and~\eqref{e:strong-phi},  and the last one from~\eqref{2-var-lim}.
Hence, we conclude that
$$
w_n \to w \ \  \text{in $L^2 (0,T;H)$, \ whence} \quad w_n \to w \ \
\aein \ \Omega \times (0,T)
$$
(the latter convergence holding up to a subsequence). By continuity
of $\alpha$, we also have $\alpha(w_n) \to \alpha (w)$ a.e. in
$\Omega \times (0,T)$. Estimate~\eqref{e:est-refined_1} (recall
\eqref{e:really-necessary}) and again Theorem~\ref{t:ds} yield, for
instance, that
$$
\alpha(w_n) \to \alpha(w) \quad \text{in $L^1 (0,T; L^1 (\Omega))$,}
$$
whence the desired equality~\eqref{final-step}.
 \fin
\subsection{Proofs of Theorem \ref{th:2} and Proposition~\ref{prop:regularized}}
\label{s:3.3}
\paragraph{Proof of Theorem \ref{th:2}.}
 Let $\{ (\chi_n, w_n )\} $ be some sequence of
approximate solutions to Problem~\ref{p:2}. Thanks to
estimates~\eqref{e:stima1}, \eqref{e:stima1-delta},
\eqref{est:5-delta}, \eqref{e:stima3-delta}, \eqref{e:3.19},
and~\eqref{e:3.22}, applying standard compactness and weak
compactness results (see~\cite{simon}),
 we find a triplet
$(\chi,w,\bar{\alpha})$ such
 that, along a (not relabeled)
subsequence, the following strong, weak, and weak$^*$ convergences
hold as $n \to +\infty$:
\begin{align}
 \label{e:conv-chi-strong-v}
 &
\begin{aligned}
  \chi_n \to \chi \quad &\text{in \ $L^2
(0,T; H^{2-\eps} (\Omega)) \cap L^q (0,T;V) \cap \CC^0 ([0,T];
H^{1-\eps} (\Omega))$} \\ &  \text{for  every $\eps >0 $ and $1 \leq
q < +\infty$,}
\end{aligned}
\\
& \label{e:conv-chi-weak-v}
\begin{aligned}
  \chi_n \weaksto \chi \quad &  \text{in \ $L^2
(0,T; Z) \cap L^\infty (0,T;V) \cap H^1 (0,T;H)$,}
\end{aligned}
\\
& \label{e:conv-w-strong-v} w_n  \weaksto w \quad \text{in \ $L^2
(0,T; V) \cap L^{2p+2} (0,T; L^{\infty}(\Omega))$.}
\end{align}
In particular, from~\eqref{e:conv-chi-weak-v}, we deduce that
\begin{equation}
\label{step-1} \mn(\chi_{n,t}) \weakto \mn(\chi_t) \qquad \text{in
$L^2 (0,T; Z)$.}
\end{equation}
Furthermore,  by \eqref{e:to-be-quoted}, there exists $\bar{\alpha}
\in L^{\rho_p} (0,T; L^{\kappa_p} (\Omega)) $ such that
\begin{equation}
\label{e:conv-alpha-weak-v} \alpha(w_n) \weakto \bar{\alpha} \quad
\text{in \ $L^{\rho_p} (0,T; L^{\kappa_p} (\Omega))$.}
\end{equation}
Now, up to a subsequence, by the last of~\eqref{e:conv-chi-strong-v}
and  by continuity of $\phi$, we have, for all $t \in [0,T]$,
\begin{equation}
\label{step0} \phi(\chi_n (\cdot,t)) \to \phi(\chi(\cdot,t)) \qquad
\aein \  \Omega\,.
\end{equation}
 On the other
hand, it follows from estimate~\eqref{e:useful-again} that
\begin{equation}
\label{step00}
  \text{the sequence $\{ \phi(\chi_n )\}$ is
uniformly integrable in $L^1 (0,T;L^1 (\Omega))$.}
\end{equation}
 Then, by~\eqref{step0}--\eqref{step00} and
Theorem~\ref{t:ds}, we conclude
 that, along the same subsequence  as in~\eqref{step0}, $\phi(\chi_n) \to
 \phi(\chi)$ in $L^1 (0,T; L^1 (\Omega))$. We
 then have, up to a subsequence,
 \begin{equation}
\label{step1} \phi(\chi_n (t)) \to \phi(\chi(t)) \qquad \text{in
$L^1 (\Omega)$}  \ \ \forae\ t \in (0,T)\,.
 \end{equation}
Next,  using~\eqref{est:3-delta}, we see that
 $\phi(\chi_n)$  is
 uniformly integrable in $L^{\nu} (0,T; L^1 (\Omega))$ for all $\nu \in [1,+\infty)$.
 Applying Theorem~\ref{t:ds}, from \eqref{step1},
 we deduce that
  \begin{equation}
\label{step2} \phi(\chi_n ) \to \phi(\chi) \qquad \text{in $L^{\nu}
(0,T; L^1(\Omega))$ \ for every $\nu \in [1,+\infty)$.}
 \end{equation}
Collecting~\eqref{e:conv-chi-strong-v}--\eqref{e:conv-alpha-weak-v}
and \eqref{step2}, we conclude that the triplet
$(\chi,w,\bar{\alpha})$ satisfies
\begin{align}
& \label{1-var-lim-v} \chi_t + A \bar{\alpha} =0 \qquad \aein \
\Omega \times (0,T)\,,
\\
& \label{2-var-lim-v} \delta \chi_t + A\chi + \phi(\chi) =w \qquad
\aein \  \Omega \times (0,T)\,.
\end{align}
It remains to show  that $\bar{\alpha} \equiv \alpha(w)$. To this
aim, we note that $\alpha$ defines a maximal monotone graph in the
duality $(L^{2p+2} (\Omega\times (0,T)), L^{\rho_p} (\Omega\times
(0,T)))$ (note that $\rho_p$ and $2p+2$ are conjugate exponents).
Taking into account relations \eqref{e:conv-w-strong-v}
and~\eqref{e:conv-alpha-weak-v}, and applying  a well-known result
from the theory of maximal monotone operators in Banach spaces
(see~\cite[Lemma~1.3, p.~42]{barbu}), it is then sufficient to prove
that
\begin{equation}
\label{3.40} \limsup_{n\to +\infty} \int_{0}^T \int_{\Omega}
\alpha(w_n) w_n \leq \int_{0}^T \int_{\Omega} \bar{\alpha} w\,.
\end{equation}
Now,
\begin{equation}
\label{a1}
\begin{aligned}
\int_{0}^T \int_{\Omega} \alpha(w_n) w_n & =\int_{0}^T \int_{\Omega}
\big(\alpha(w_n)-\mo(\alpha(w_n))\big)\, w_n +|\Omega| \int_{0}^T
\mo(\alpha(w_n))\, \mo(w_n)\\ & = -\int_{0}^T\int_{\Omega}
{w_n}\,{\mn(\chi_{n,t})} +|\Omega| \int_{0}^T \mo(\alpha(w_n))\,
\mo(w_n)\,,
\end{aligned}
\end{equation}
where the second equality follows from~\eqref{1-var-better}. Then,
using~\eqref{Aa} and \eqref{2-var-better}, we find the chain of
inequalities
\begin{equation}
\label{e:3.40}
\begin{aligned}
 \liminf_{n \to +\infty}
& \Big( \int_{0}^T \int_{\Omega}{w_n}\,{\mn(\chi_{n,t})}\Big)
\\
&  \geq  \liminf_{n \to +\infty} \delta \int_{0}^T \| \chi_{n,t}
\|_{V'}^2 + \lim_{n \to +\infty} \int_{0}^T \int_{\Omega} \chi_{n,t}
\chi_n  + \lim_{n \to +\infty} \int_{0}^T \int_{\Omega}\phi(\chi_n)
\mn(\chi_{n,t})\\ & \geq \delta \int_{0}^T \| \chi_{t} \|_{V'}^2 +
\int_0^T \int_{\Omega} \chi_t \chi + \int_{0}^T
\int_{\Omega}\phi(\chi) \mn(\chi_{t})\\ &= \int_{0}^T \int_{\Omega}
{w}\,{\mn(\chi_{t})} =-\int_{0}^T \int_{\Omega}
\big(\bar{\alpha}-\mo(\bar{\alpha})\big)\, w\,,
\end{aligned}
\end{equation}
where the second inequality follows from
convergences~\eqref{e:conv-chi-strong-v}
and~\eqref{e:conv-chi-weak-v} for $\chi_n$ and from
combining~\eqref{step-1} with~\eqref{step2}, while the subsequent
identities are due to~\eqref{1-var-lim-v}--\eqref{2-var-lim-v}. On
the other hand, it follows from~\eqref{e:conv-alpha-weak-v} that
\begin{equation}
\label{step3} \mo(\alpha(w_n)) \weakto \mo(\bar{\alpha}) \qquad
\text{in $L^{\rho_p} (0,T)$,}
\end{equation}
 whereas, from~\eqref{step2}, we gather
that
\begin{equation}
\label{step4} \mo(w_n)= \mo(\phi (\chi_n)) \to
\mo(\phi(\chi))=\mo(w) \qquad \text{in $L^{2p+2} (0,T)$.}
\end{equation}
Combining \eqref{step3}--\eqref{step4}, we conclude that
\begin{equation}
\label{e:3.41} \lim_{n \to +\infty} |\Omega| \int_{0}^T
\mo(\alpha(w_n))\, \mo(w_n)=
 |\Omega| \int_{0}^T \mo(\bar{\alpha})\, \mo(w)\,.
\end{equation}
Collecting~\eqref{a1}, \eqref{e:3.40},  and \eqref{e:3.41}, we infer
the desired~\eqref{3.40}.  Ultimately, we have proved that
\begin{equation}
\label{e:for-later-convenience}
 \alpha(w_n) \weakto \alpha(w) \ \
\text{in \ $L^{\rho_p} (0,T; L^{\kappa_p} (\Omega))$  \ \ and} \ \
\lim_{n\to +\infty} \int_{0}^T \int_{\Omega} \alpha(w_n) w_n =
\int_{0}^T \int_{\Omega} \alpha(w) w\,.
\end{equation}
 \fin 
\paragraph{Proof of Proposition~\ref{prop:regularized}.}\,
In order to prove that
system~\eqref{1-var-better}--\eqref{2-var-better} enjoys the
regularization in
time~\eqref{e:further-reg-chi}--\eqref{e:further-reg-w},   using the
Gagliardo-Nirenberg interpolation inequality we note that
 $L^2(0,T; Z) \cap L^\infty (0,T;V) \subset L^8 (0,T; W^{1,12/5}(\Omega)) $
 with continuous embedding. Therefore,
 regularity~\eqref{reg-chi-bis} for $\chi$ and standard Sobolev embeddings yield
\begin{equation}
\label{e:stichi} \| \chi\|_{L^8 (0,T; L^{12}(\Omega))} \leq C\,.
\end{equation}

Now, we test \eqref{2-var-better} by $A\chi_t$. 
Note that all the forthcoming computations are rigorous on the
ap\-pro\-xi\-ma\-tion scheme for Problem~\ref{p:2} which we will
detail in  Appendix.
 Elementary calculations yield
\begin{equation}
\label{e:elem1}
\begin{aligned}
    \frac{\dd}{ \dd
t}\left(\frac12 \int_{\Omega}|A\chi|^2 \right) +
\delta\int_{\Omega}|\nabla \chi_t|^2 = I_1 + I_2\,,
\end{aligned}
\end{equation}
with
\begin{align}
&
 \label{e:elem2}    I_1  :=  \int_{\Omega} \nabla w
\,\cdot\, \nabla \chi_t \leq \frac{\delta}2\int_{\Omega} |\nabla
\chi_t|^2 + \frac{1}{2\delta}\int_{\Omega} |\nabla w|^2\,,
\\
&  \label{e:elem3}
\begin{aligned}
 I_2   &  := -  \int_{\Omega}  \phi{'}(\chi) \left(\nabla\chi
 \cdot \nabla \chi_t\right)
\\ & \leq C_{\phi,3} \int_{\Omega}|\nabla \chi||\nabla
\chi_t| \left( 1+|\chi|^4\right)\\ &  \leq C \| \nabla \chi_t \|_{H}
\| \nabla \chi \|_{L^6(\Omega)} \left(\| \chi\|_{L^{12}(\Omega)}^4 +
1 \right)\\ &  \leq \frac{\delta}{4}\| \nabla \chi_t \|_{H}^2 + C
\left(\| \chi\|_{L^{12}(\Omega)}^8 + 1 \right) \| \chi \|_{Z}^2\,,
\end{aligned}
\end{align}
 where the second inequality follows from~\eqref{e:addphi}, the
 third one from the H\"{o}lder inequality, and the last one by taking
 into account the continuous embedding $Z \subset W^{1,6}(\Omega)$.

Collecting~\eqref{e:elem1}--\eqref{e:elem2}, taking into
account~\eqref{e:stichi},  and applying the uniform
 Gronwall Lemma (see \cite[Lemma~III.1.1]{temam}), we find for every $\tau>0$ an
estimate of the form \eqref{uniform-gronwall} for $\nabla \chi_t $
in $L^2 (\tau,T;H) $ and for $A\chi $ in $L^\infty (\tau,T;H)$,
whence
\[
\chi \in L^\infty (\tau,T;Z) \cap H^1 (\tau,T;V)\quad \text{for all
} 0<\tau<T\,.
\]
 Then, a comparison in~\eqref{1-var-better} also yields a bound
for $A(\alpha(w))$ in $L^{\rho_p}(\tau,T;V)$, whence an estimate for
$\alpha(w) $ in $L^{\rho_p}(\tau,T;H^3(\Omega))$, in view
of~\eqref{reg-w}.
 Thus, we conclude
\eqref{e:further-reg-chi}--\eqref{e:further-reg-w}, as well as
estimate \eqref{uniform-gronwall}. \fin

\section{Global attractor for Problem~\ref{p:2}}
\label{s:5}
\subsection{Proof of Proposition~\ref{prop:2.1}}
\label{ss:5.1} We need two preliminary lemmas. The first one
clarifies some properties of the energy
functional~$\ene$~\eqref{e:ene-funct}.
\begin{lemma}
\label{l:5.1} Assume~\eqref{hyp:5}--\eqref{e:hyp2-bis-visco}. Then,
the functional $\ene: \cx \to \R$ defined by~\eqref{e:ene-funct} is
bounded from below, lower-semicontinuous w.r.t. the $H$-topology,
and  satisfies the chain rule
\begin{equation}
\label{e:l5.1}
\begin{gathered}
\text{for all $v \in H^1(0,T;H)$ with $Av+ \phi(v)  \in
L^2(0,T;H)$,} \\
\text{ the map $t \in [0,T] \mapsto \ene(v(t))$ is absolutely
continuous, and}
\\
\frac{\mathrm{d}}{\mathrm{d}t} \ene(v(t)) = \int_{\Omega}  v_t(t)
\left(Av(t)+ \phi(v(t)) \right) \qquad \forae\ t \in (0,T)\,.
\end{gathered}
\end{equation}
\end{lemma}
\begin{proof}
In order to prove the lower-semicontinuity property, we fix a
sequence $\{v_n \}$ converging to some $v $ in $H$ and assume,
without loss of generality, that $\sup_n \ene(v_n) <+\infty$.
 Since
$\widehat{\phi}$ is bounded from below, we conclude that $\{v_n\}$
is actually bounded in $V$, and thus $v_n \weakto v$ in $V$,
yielding $\textstyle \int_{\Omega} |\nabla v|^2  \leq \liminf_{n}
\int_{\Omega} |\nabla v_n|^2$. On the other hand,
\[
\begin{aligned}
\liminf_{n \to +\infty} \int_{\Omega} \widehat{\phi}(v_n) &  =
\liminf_{n \to +\infty} \int_{\Omega}
\left(\widehat{\phi}(v_n)+\frac{C_{\phi,2}}{2} |v_n|^2 \right) -
\frac{C_{\phi,2}}2\lim_{n \to +\infty} \int_{\Omega}|v_n|^2  \\ &
\geq \int_{\Omega}\left( \widehat{\phi}(v)+\frac{C_{\phi,2}}2
|v|^2\right) -\frac{C_{\phi,2}}2\int_{\Omega}|v|^2\,, \end{aligned}
\]
the latter inequality following from~\eqref{e:lambda-convex}  and, for
instance, from Ioffe's Theorem~\ref{th-ioffe}.
 Finally, to check the chain
rule~\eqref{e:l5.1}, we observe that the functional
\begin{equation}
\label{funz-bar} \mathcal{E}_{\mathrm{cv}} (v): = \mathcal{E}(v)  +
\frac{C_{\phi,2}}{2}\int_{\Omega} |v|^2 \qquad \text{for all $v \in
X$}
\end{equation}
is convex, thanks to~\eqref{e:lambda-convex}. Then, \eqref{e:l5.1}
follows from the chain rule for $ \mathcal{E}_{\mathrm{cv}}$,
see~\cite[Lemma~III.3.3]{brezis73}.
\end{proof}
\begin{lemma}
\label{l:5.2} Assume~\eqref{e:hyp1}. Then,
\begin{equation}
\label{e:5.2.2} \text{for all $w \in V \cap L^\infty (\Omega)$, there
holds} \ \ \nabla \alpha(w(x)) = \alpha'(w(x)) \nabla w(x) \ \
\forae\ x \in \Omega\,.
\end{equation}
\end{lemma}
\begin{proof}
Since $\Omega$ is smooth, we can take a sequence $\{w_k \}\subset
\mathrm{C}^1 (\overline{\Omega})$ such that $w_k \to w$ in $V \cap
L^q (\Omega)$ for all $1\leq q<+\infty$. Clearly, for all $k \in \N$,
there holds
\begin{equation}
\label{e:clearl} \nabla \alpha(w_k(x)) = \alpha'(w_k(x)) \nabla
w_k(x) \qquad \forall\, x \in \Omega\,.
\end{equation}
Now, since $\alpha'(r)$  grows like $|r|^{2p}$ by~\eqref{e:hyp1}, we
conclude that $\alpha'(w_k)\to \alpha'(w) $ and $\alpha(w_k)\to
\alpha(w) $  in $L^q (\Omega)$ for all $1\leq q<+\infty$. Therefore,
$\nabla \alpha(w_k) = \alpha'(w_k) \nabla w_k \to \alpha'(w) \nabla
w $ in $L^{\rho}(\Omega)$ for all $\rho\in[1,2)$ and~\eqref{e:5.2.2}
follows.
\end{proof}
\paragraph{Proof of Proposition~\ref{prop:2.1}.}
Thanks to~Theorem~\ref{th:2}, the set $\sfl$ complies with the
existence axiom~\textbf{(P1)} in
Definition~\ref{def:generalized-semiflow}. The translation
property~\textbf{(P2)} is immediate to check. Concerning the
concatenation axiom, let $\chi_1$ and $\chi_2$ be two solutions to
Problem~\ref{p:2} on $(0,+\infty)$, satisfying $\chi_1(\tau) =
\chi_2 (0)$ for some $\tau \geq 0$, and let the functions $w_1$ and
$w_2$ be such that,  for $i=1,2$,  the pairs $(\chi_i,w_i)$ satisfy
equations~\eqref{1-var-better}--\eqref{2-var-better}, with
regularities~\eqref{reg-chi-bis} and~\eqref{reg-w}. Then, one easily
sees that the concatenations (cf.~\eqref{def:concaten}) $\tilde
\chi$ and $\tilde w$ of $\chi_1, \, \chi_2$ and $w_1,\, w_2$,
respectively, satisfy
equations~\eqref{1-var-better}--\eqref{2-var-better},  and still
enjoy regularities~\eqref{reg-chi-bis} and~\eqref{reg-w}, respectively (the
fact that $\chi_1(\tau) = \chi_2 (0)$ is crucial for the
time-regularity of $\tilde{\chi}$).

To prove that all solutions $\chi \in \sfl$ are continuous w.r.t.
the phase space topology~\eqref{e:pspace}, let us fix $\{t_n\}, t_0$
in $[0,+\infty)$, and show that
\begin{equation}
\label{e:continuity}
 t_n \to t_0 \ \Rightarrow \ \left( \| \chi(t_n)
-\chi(t_0)\|_{V} + \left\|\widehat{\phi}(\chi(t_n))-
\widehat{\phi}(\chi(t_0)) \right\|_{L^1(\Omega)} \right)\to 0\quad
\text{as $n \to +\infty$}\,.
\end{equation}
Indeed, thanks to regularity~\eqref{reg-chi-bis}, for all $T>0$, the
function $\chi:[0,T] \to V$ is continuous w.r.t. the weak
$V$-topology, hence
\begin{equation}
\label{e:weakV} \chi(t_n) \weakto \chi(t_0) \qquad \text{in $V$.}
\end{equation}
Therefore, by Lemma~\ref{l:5.1}, we have
\[
\liminf_{n\to +\infty} \mathcal{E}(\chi(t_n)) \geq
\mathcal{E}(\chi(t_0))\,.
\]
Combining this inequality with
  the continuity of the map $t \in [0,T] \mapsto
\mathcal{E}(\chi(t))$, one concludes that
\begin{subequations}
\label{e:converg}
\begin{align}
& \label{e:c1}
 \lim_{n \to +\infty} \int_{\Omega} |\nabla
\chi(t_n)|^2 = \int_{\Omega} |\nabla \chi(t_0)|^2,
\\
& \label{e:c2}
 \lim_{n \to +\infty}
\int_{\Omega} \widehat{\phi}(\chi(t_n))= \int_{\Omega}
\widehat{\phi}(\chi (t_0))\,.
\end{align}
\end{subequations}
Clearly, \eqref{e:weakV}, combined with \eqref{e:c1}, yields that
$\chi(t_n) \to \chi(t_0)$ in $V$. In order to prove the
additional convergence
\begin{equation}
\label{e:stronger-conv} \|\widehat{\phi}(\chi(t_n))-
\widehat{\phi}(\chi(t_0)) \|_{L^1 (\Omega)} \to 0 \quad \text{as $n
\to +\infty$}\,,
\end{equation}
we note that \eqref{e:continuity} implies, in particular, that
\begin{equation}
\label{e:ci-serve} \widehat{\phi}(\chi(\cdot, t_n)) \to
\widehat{\phi}(\chi(\cdot,t_0)) \qquad \aein\ \Omega\,.
\end{equation}
In view of~\cite[Lemma~4.2]{rocca-schimperna04}, \eqref{e:ci-serve},
combined with \eqref{e:c2} and the fact that $\widehat{\phi}$ takes
non-negative values,  yields~\eqref{e:stronger-conv}.

 The  energy identity~\eqref{e:enid} follows
by multiplying~\eqref{1-var-better} by $w$ (note that the latter
is an admissible test function, thanks to~\eqref{reg-w}),
\eqref{2-var-better} by $\chi_t$, adding the resulting
relations,  taking into account the chain rule~\eqref{e:l5.1} and
formula~\eqref{e:5.2.2}, and integrating in time.

 It remains to
prove the upper-semicontinuity with respect to the initial data. To this
aim, we will exploit~\eqref{e:enid}.  Thus,
 let us fix a sequence of solutions $\{\chi_n \} \subset \sfl$
 and $\chi_0 \in \cx$,
 with
\begin{equation}
\label{conv-zero}
  \dcx(\chi_n(0),\chi_0) \to 0  \ \text{as $n \to +\infty$, so that, in particular, $\ene(\chi_n(0)) \to
  \ene(\chi_0)$.}
  \end{equation}
 Identity~\eqref{e:enid} yields  that there exists a constant $C>0$ such that,
 for all $n \in \N$,
 \begin{equation}
 \label{aprio1}
  \delta\int_0^t \int_{\Omega} |\partial_t \chi_n|^2 + \int_0^t
\int_{\Omega} \alpha'(w_n)|\nabla w_n|^2 + \ene(\chi_n(t)) =
\ene(\chi_n(0)) \leq C \ \ \text{for all $t\geq 0$}.
\end{equation}
Arguing as in Section~\ref{s:3.1}, we obtain
estimates~\eqref{e:stima1}, \eqref{e:stima1-delta},
\eqref{est:5-delta}, \eqref{e:stima3-delta}, \eqref{e:3.19},
and~\eqref{e:3.22}  for the sequence $\{(\chi_n,w_n)\} $, on every
interval $(0,T)$. Therefore, with a diagonalization procedure, we
find a subsequence $\{(\chi_{n_k}, w_{n_k}) \}$ and  functions
$(\chi,w):(0,+\infty) \to \cx \times V$ for
which~\eqref{e:conv-chi-strong-v}--\eqref{e:conv-w-strong-v},
\eqref{step2}, and \eqref{e:for-later-convenience} hold on every
interval $(0,T)$,  for all $T>0$.
 Using all the aforementioned relations, we have
$\chi(0)=\chi_0$ and,  arguing as in Section~\ref{s:3.3}, we
conclude that $\chi \in \sfl$. In order to prove that
\begin{equation}
\label{e:conv-chink} \text{for all $t \geq 0$}, \ \  \left( \|
\chi_{n_k}(t) -\chi(t)\|_{V} + \left\|
\widehat{\phi}(\chi_{n_k}(t))- \widehat{\phi}(\chi(t)) \right\|_{L^1
(\Omega)} \right)\to 0 \ \ \ \text{as $k \to +\infty$,}
\end{equation}
we first obtain some \emph{enhanced} convergence for
the sequence $\{ w_{n_k}\}$. To this aim, we note that, for every
$T>0$, there holds
\[
\begin{aligned}
\mathsf{c}_\alpha \limsup_{k \to +\infty} \int_0^T \int_{\Omega}
|w_{n_k}|^{2p+2} & \leq \limsup_{k \to +\infty}   \int_0^T
\int_{\Omega}  \alpha(w_{n_k}) w_{n_k} - \liminf_{k \to +\infty}
\int_0^T \int_{\Omega} \Psi(w_{n_k}) \\ & \leq  \int_0^T
\int_{\Omega} \alpha(w) w -  \int_0^T \int_{\Omega} \Psi(w)  =
\mathsf{c}_\alpha \int_0^T \int_{\Omega} |w|^{2p+2}\,.
\end{aligned}
\]
Indeed, the first inequality follows from~\eqref{e:hyp-add-alpha},
the second one from the second convergence
in~\eqref{e:for-later-convenience}, and
from~\eqref{e:conv-w-strong-v}, together with the convexity of $\Psi$
(thanks to  Ioffe's Theorem~\cite{ioffe77}), and the third one
from~\eqref{e:hyp-add-alpha} again. Taking into account the fact that
\[
\liminf_{k \to +\infty} \int_0^T \int_{\Omega} |w_{n_k}|^{2p+2}  \geq
\int_0^T \int_{\Omega} |w|^{2p+2},
\]
due to~\eqref{e:conv-w-strong-v}, we have
\[
w_{n_k} \to w \qquad \text{in $L^{2p+2}(0,T; L^{2p+2} (\Omega))$ \ \
for all $T>0$}
\]
and, thus, finally,
\begin{equation}
\label{strong-w} w_{n_k} \to w \qquad \text{in measure in $\Omega
\times (0,T)$ for all $T>0$.}
\end{equation}
 As a consequence, for all $t \geq 0$,
\begin{equation}
\label{e:ioffe-again} \liminf_{k \to +\infty} \int_0^t \int_{\Omega}
\alpha'(w_{n_k}) |\nabla w_{n_k}|^2 \geq \int_0^t \int_{\Omega}
\alpha'(w) |\nabla w|^2\,,
\end{equation}
thanks to the convergence in measure~\eqref{strong-w}, the weak
 convergence~\eqref{e:conv-w-strong-v} for $\{ \nabla w_{n_k}\}$ in
 $L^2 (0,T; H)$ for all $T>0$, and again  Ioffe's Theorem~\ref{th-ioffe}.
 Hence, passing to the limit in  the energy identity~\eqref{aprio1} (written for the functions
 $(\chi_{n_k},w_{n_k})$),  we infer, for all $t \geq 0$,
\begin{equation}
\label{e:elem-arg}
\begin{aligned}
 \delta\int_0^t \int_{\Omega} |\partial_t \chi|^2  & + \int_0^t
\int_{\Omega} \alpha'(w)|\nabla w|^2 + \ene(\chi(t)) \\ & \leq
\liminf_{k \to +\infty} \left(  \delta\int_0^t \int_{\Omega}
|\partial_t \chi_{n_k}|^2 + \int_0^t \int_{\Omega}
\alpha'(w_{n_k})|\nabla w_{n_k}|^2 + \ene(\chi_{n_k}(t))  \right)
\\ & \leq
\limsup_{k \to +\infty} \left(  \delta\int_0^t \int_{\Omega}
|\partial_t \chi_{n_k}|^2 + \int_0^t \int_{\Omega}
\alpha'(w_{n_k})|\nabla w_{n_k}|^2 + \ene(\chi_{n_k}(t))  \right)
\\
&  = \lim_{k \to +\infty} \ene(\chi_{n_k}(0))= \ene(\chi_0) =
\delta\int_0^t \int_{\Omega} |\partial_t \chi|^2 + \int_0^t
\int_{\Omega} \alpha'(w)|\nabla w|^2 + \ene(\chi(t))\,,
\end{aligned}
\end{equation}
where the first inequality follows
from~\eqref{e:conv-chi-strong-v}--\eqref{e:conv-chi-weak-v},
\eqref{e:ioffe-again}, and the fact that $\ene$ is
lower-semicontinuous w.r.t. the $H$-topology, the  third one
from~\eqref{aprio1}, the fourth one from~\eqref{conv-zero}, and the
last equality from  the \emph{energy identity}~\eqref{e:enid}
satisfied by all solutions in $\mathcal{S}$. With an elementary
argument, we deduce from~\eqref{e:elem-arg} that, for all $t>0$,
\[
\int_0^t \int_{\Omega} |\partial_t \chi_{n_k}|^2 \to \int_0^t
\int_{\Omega} |\partial_t \chi|^2, \quad \text{whence} \quad
\chi_{n_k} \to \chi \ \ \text{in $H^1(0,t;H)$,}
\]
as well as
\[
\ene(\chi_{n_k}(t))  \to \ene(\chi(t)).
\]
 Arguing in the same way as
 throughout~\eqref{e:weakV}--\eqref{e:ci-serve} and again
 invoking~\cite[Lemma~4.2]{rocca-schimperna04}, we
 obtain~\eqref{e:conv-chink}.
  This concludes the
proof.
\fin 
\subsection{Proof of Theorem~\ref{th:4}}
\label{ss:5.2}
\paragraph{Eventual boundedness.}
In order to check that
 $\sfl$ is eventually bounded, we fix  a ball  $B(0,R)$ centered at $0$
   of radius $R$ in $\cx$, some initial datum
$\chi_0 \in B_\cx(0,R)$, namely
 satisfying (recall that we can assume that $\widehat{\phi}$ is a positive function)
\begin{equation}
  \label{e:in-a-ball}
   \|\chi_0\|_V+\int_{\Omega}\widehat{\phi}(\chi_0) \leq
 R,
\end{equation}
  and consider a generic trajectory
      $\chi \in \sfl$
   starting from $\chi_0$.
 Recalling the energy identity~\eqref{e:enid}, we
 find, for all $t \geq 0$,
\begin{equation}
\label{e:eb1} \int_{\Omega}\widehat{\phi}(\chi(t)) \leq
\ene(\chi(t)) \leq \ene(\chi_0) \leq R, \qquad \int_{\Omega}|\nabla
\chi(t)|^2 \leq 2R\,.
\end{equation}
 Now,
taking into account the fact that $m(\chi(t)) = m(\chi_0)$
 for all $t \geq 0$ (cf. \eqref{e:constant-m-value}), we deduce
 from \eqref{e:eb1} a bound for $\| \chi \|_{L^\infty
 (0,+\infty;V)}$.
  Hence,
 there
exists $R'>0$ such that $\dcx(\chi(t),0) \leq R'$ for all $t \geq
0$. Since $\chi_0$ is arbitrary, we conclude that
 the evolution of the ball $B_\cx(0,R)$ is
contained in the ball $B_\cx(0,R')$.
\paragraph{Compactness.}
In order to  verify that $\sfl$ is compact, we  consider  a sequence
$\{\chi_n\}\subset\sfl$ such that $\{\chi_n(0)\}$ is bounded in
$\cx$. We write the energy identity~\eqref{aprio1} and, as in the
proof of Proposition~\ref{prop:2.1}, deduce that there exist a
 subsequence
$\{(\chi_{n_k}, w_{n_k}) \}$ and  functions $(\chi,w):(0,+\infty)
\to \cx \times V$ for which
convergences~\eqref{e:conv-chi-strong-v}--\eqref{e:conv-w-strong-v},
\eqref{step2}, and \eqref{e:for-later-convenience} hold on every
interval $(0,T)$  for all $T>0$. However, we cannot  prove that
\begin{equation}
\label{e:conv-chink-comp}   \left( \| \chi_{n_k}(t) -\chi(t)\|_{V} +
\left\|\widehat{\phi}(\chi_{n_k}(t))- \widehat{\phi}(\chi(t))
\right\|_{L^1(\Omega)} \right)\to 0 \ \ \ \text{for all $t>0$},
\end{equation}
arguing in the same way
 as throughout~\eqref{e:conv-chink}--\eqref{e:elem-arg}, for, in
this case, we do not have the convergence of the initial
energies $\ene(\chi_{n_k}(0))$ at our disposal. Then, we rely on the following
procedure (see also \cite{rossi-segatti-stefanelli08, segatti06} for
the use  of an analogous argument).

First, we apply  Helly's compactness principle (with respect to
the~pointwise convergence) for monotone functions to the functions
$t \mapsto \ene(\chi_{n_k}(t))$, which are non-increasing in view of
the energy identity~\eqref{aprio1}. Thus, up to a (not relabeled)
subsequence, there exists a non-increasing function $\mathscr{E}:
[0,+\infty) \to \R$ such that
\begin{equation}
\label{helly} \mathscr{E}(t):= \lim_{k \to +\infty}
\ene(\chi_{n_k}(t)) \qquad \text{for all $t \geq 0$}.
\end{equation}
By the lower-semicontinuity of $\ene$ (w.r.t. the $H$-topology), we find

\begin{equation}
\label{helly-ineq} \ene(\chi(t)) \leq \mathscr{E}(t) \qquad
\text{for all $t \geq 0$.}
\end{equation}
On the other hand, \eqref{e:conv-chi-strong-v} ensures that, up to a
further extraction, for almost all $s \in (0,t)$,
\begin{equation}
\label{e:enhanced-conv} \chi_{n_k}(s) \to \chi(s) \ \  \text{in
$H^{2-\eps}(\Omega)$ for all $\eps>0$, whence} \ \ \chi_{n_k}(s) \to
\chi(s) \ \ \text{in $H^1 (\Omega) \cap L^\infty(\Omega) $.}
\end{equation} Thus, in particular,
\begin{equation}
\label{e:pointwise} \widehat{\phi}(\chi_{n_k}(\cdot,s)) \to
\widehat{\phi}(\chi(\cdot,s)) \qquad \aein \ \Omega\,.
\end{equation} Moreover, for every $\mathcal{O} \subset \Omega$, there
holds
\begin{equation}
\label{e:unifinte}
\begin{aligned}
 \int_{\mathcal{O}}
|\widehat{\phi}(\chi_{n_k}(s))|  & \leq
|\mathcal{O}||\widehat{\phi}(0)| +\frac{C_{\phi,2}}2
\int_{\mathcal{O}} |\chi_{n_k}(s)|^2 + \int_{\mathcal{O}}
|{\phi}(\chi_{n_k}(s))| |\chi_{n_k}(s)| \\ &  \leq C
\left(|\mathcal{O}|+
\int_{\mathcal{O}}|{\phi}(\chi_{n_k}(s))|\right)\,,
\end{aligned}
\end{equation} where the first inequality follows
from~\eqref{e:phi-add} and the second one
from~\eqref{e:enhanced-conv}. Notice that the right-hand side
of~\eqref{e:unifinte} tends to zero as $|\mathcal{O}| \to 0$, since
the sequence $\{ {\phi}(\chi_{n_k}(s)) \}$ is uniformly integrable
in $L^1 (\Omega)$ thanks to~\eqref{step2}. Hence, \eqref{e:unifinte}
yields that $\{ \widehat{\phi}(\chi_{n_k}(s)) \}$ is itself
uniformly integrable in $L^1 (\Omega)$. Combining this
with~\eqref{e:pointwise}, in view of Theorem~\ref{t:ds} we conclude
that $\widehat{\phi}(\chi_{n_k}(s)) \to \widehat{\phi}(\chi(s))$ in
$L^1 (\Omega)$. Finally, we have shown that
 there exists a negligible set $\mathscr{N} \subset (0,+\infty)$ such
that
\begin{equation}
\label{e:convene} \mathscr{E}(s) = \lim_{k \to +\infty}
\ene(\chi_{n_k}(s)) = \ene(\chi(s)) \qquad \forae\, s \in
(0,+\infty) \setminus \mathscr{N}\,.
\end{equation}

We are now in a position to carry out the argument
for~\eqref{e:conv-chink-comp} (which bypasses the lack of
convergence of the initial data in the phase
space~\eqref{e:pspace}), using the fact that the energy
identity~\eqref{aprio1} holds for all $t>0$. Indeed, for every fixed
$t>0$ and for all $s \in (0,t) \setminus \mathscr{N}$, we can
 pass to the limit in the energy identity~\eqref{aprio1}, written
for the sequence $(\chi_{n_k}, w_{n_k})$ on the interval $(s,t)$.
Note  indeed that
convergences~\eqref{e:conv-chi-strong-v}--\eqref{e:conv-w-strong-v}
and \eqref{step2} for $(\chi_{n_k}, w_{n_k})$   hold on $(s,t)$.
 Proceeding as above, we then deduce once more
that
\[
\lim_{k \to +\infty} \int_s^t \int_{\Omega} \alpha(w_{n_k}) w_{n_k} =
\int_s^t \int_{\Omega} \alpha(w) w\,,
\]
whence $w_{n_k} \to w$ in $L^{2p+2}(s,t;L^{2p+2}(\Omega))$.
 Therefore,
 repeating the very
same passages  as in~\eqref{e:elem-arg} and  relying
on~\eqref{e:convene}, we find
\[
\begin{aligned}
 \delta\int_s^t \int_{\Omega} |\partial_t \chi|^2  & + \int_s^t
\int_{\Omega} \alpha'(w)|\nabla w|^2 + \ene(\chi(t)) \\ & = \lim_{k
\to +\infty} \left(\delta\int_s^t \int_{\Omega} |\partial_t
\chi_{n_k}|^2 + \int_s^t \int_{\Omega} \alpha'(w_{n_k})|\nabla
w_{n_k}|^2 + \ene(\chi_{n_k}(t))  \right)\,,
\end{aligned}
\]
which gives
\[
\mathscr{E}(t) = \lim_{k \to +\infty} \ene(\chi_{n_k}(t)) =
\ene(\chi(t)) \quad\text{for all $t>0$},
\]
and, finally, \eqref{e:conv-chink-comp}.

\paragraph{Lyapunov function and rest points.} We now verify that
$\ene$ acts as a Lyapunov functional for
$\sfl$. Actually, $\ene$ clearly is continuous  on $\cx$ and
decreasing along all solutions, thanks to  the energy
identity~\eqref{e:enid}. Furthermore, assume that, along some $\chi \in
\sfl$, the map $t \in [0,+\infty) \mapsto \ene(\chi(t))$ is constant.
Then, in view of~\eqref{e:enid},
 we find $\nabla w \equiv 0$ and
 $\chi_t \equiv 0$ a.e. in
$(0,+\infty)$, so that $\chi(t) \equiv \chi(0) $  for all $t \in
[0,+\infty)$. Analogously, we immediately find that $\bar{\chi} \in
\cx$ is a rest point for $\sfl$ if  and only if it satisfies the
stationary system
\begin{subequations}
\label{e:sub}
\begin{align}
& \label{e:sub1} A (\alpha(\bar{w})) =0 \quad \aein  \ \Omega\,,
\\
& \label{e:sub2} A \bar{\chi}+ \phi(\bar{\chi}) =\bar{w} \quad \aein
\
 \Omega\,.
\end{align}
\end{subequations}
\paragraph{Conclusion of the proof.} We apply
Theorem~\ref{thm:ball1} and
Remark~\ref{rem:restriction_to_invariant_set} with the choice
$\mathcal{D}:=\mathcal{D}_{m_0} $ for some $m_0>0$ (cf.~\eqref{e:fixed-mean-value}). Thanks
to~\eqref{e:constant-m-value} (recall the second a priori estimate
in Section~\ref{s:3.1}),  for all $\chi_0 \in \mathcal{D}_{m_0} $,
every solution starting from the initial datum $\chi_0$ remains in
$\mathcal{D}_{m_0} $,  so that the first condition in~\eqref{e:saab}
is satisfied. To check the second one, we fix some $\bar{\chi} \in
\rest$ with $|m(\bar{\chi})| \leq m_0 $. It follows
from~\eqref{e:sub1} and~\eqref{e:hyp1} that $\nabla \bar{w} \equiv
0$, so that $\bar{w}$ is constant in $\Omega$.
 Hence, we test~\eqref{e:sub2}  by $\bar{\chi} - m(\bar{\chi})$.
Since $\bar{w}= m(\bar{w})$, we infer  that
\begin{equation}
\label{est:rest-points}
 \| \nabla \bar{\chi} \|_H^2 + \int_{\Omega} \phi(\bar{\chi})
 (\bar{\chi}-m(\bar{\chi})) \leq 0\,.
\end{equation}
On the other hand, \eqref{lim-infty-phi} ensures that
estimate~\eqref{e:2.2.14} holds, so that there exist constants
$\mathcal{K}_{m_0}, \,\mathcal{K}_{m_0}^1>0$, only depending on
$m_0$, such that
\begin{equation}
\label{est-aggiunta} \int_{\Omega}|\phi(\bar{\chi})| \leq
\mathcal{K}_{m_0} \int_{\Omega} \phi(\bar{\chi})
 (\bar{\chi}-m(\bar{\chi})) + \mathcal{K}_{m_0}^1\,.
\end{equation}
Collecting~\eqref{est:rest-points} and~\eqref{est-aggiunta}, we
deduce that
\[
\| \nabla \bar{\chi} \|_H^2 +
\frac1{\mathcal{K}_{m_0}}\int_{\Omega}|\phi(\bar{\chi})|  \leq
\frac{\mathcal{K}_{m_0}^1}{\mathcal{K}_{m_0}}\,,
\]
whence, in particular,
\[
|m(\bar{w})|=|m(\phi(\bar{\chi}))| \leq
\frac{\mathcal{K}_{m_0}^1}{|\Omega|}.
\]
Taking into account the fact that $\nabla \bar{w}=0$ (so that $\bar{w}$ is a
constant) and that $|m(\bar{\chi})| \leq m_0 $, we conclude that
\begin{equation}
\label{e:est-rest} \exists\,\mathcal{K}_{m_0}^{2}>0\,: \ \ \forall\,
\bar{\chi} \in \rest\cap\mathcal{D}_{m_0}\ \ \| \bar{\chi} \|_{V} +
| \bar{w} | \leq \mathcal{K}_{m_0}^2\,.
\end{equation}
Thus, a comparison in~\eqref{e:sub2} and the standard elliptic
regularity estimate (cf.~also the calculations developed
throughout~\eqref{e:more-convenient_form}--\eqref{e:stima3-bis}),
yield
\begin{equation}
\label{e:later}
 \exists\,\mathcal{K}_{m_0}^{3}>0\,:  \ \ \forall\,
\bar{\chi} \in \rest\cap\mathcal{D}_{m_0} \ \ \| {\phi}(\bar{\chi})
\|_{L^6 (\Omega)}+ \| \bar{\chi}\|_{W^{2,6}(\Omega)} \leq
\mathcal{K}_{m_0}^3\,,
\end{equation}
whence, in particular, an estimate for $\bar{\chi}$ in $L^\infty
(\Omega)$. Then, using~\eqref{e:phi-add},  we readily infer that
\begin{equation}
\label{e:est-rest2} \exists\,\mathcal{K}_{m_0}^{4}>0\,: \ \
\forall\, \bar{\chi} \in \rest\cap\mathcal{D}_{m_0} \ \ \|
\widehat{\phi}(\bar{\chi}) \|_{L^6 (\Omega)}  \leq
\mathcal{K}_{m_0}^4\,.
\end{equation}
Finally, \eqref{e:est-rest} and \eqref{e:est-rest2} yield that
$\rest \cap\mathcal{D}_{m_0} $ is bounded in the phase space $\cx$,
and the existence of the global attractor follows from
Theorem~\ref{thm:ball1}.

In fact,  with the same calculations   as in the above lines, joint
with a boot-strap argument, one easily proves that
\begin{equation}
\label{rest}
 \forall\, p \in [1,+\infty)\,: \ \ \exists\,C_p>0 \ \
\bar{\chi} \in \rest\cap\mathcal{D}_{m_0} \quad \|
\bar{\chi}\|_{W^{2,p}(\Omega)} + \|
\widehat{\phi}(\bar{\chi})\|_{L^p(\Omega)} \leq C_p.
\end{equation}
  Then,
estimate~\eqref{e:enhanced-regularity}  is a straightforward
consequence of~\eqref{e:additional} and~\eqref{rest}.
 \fin

\label{s:6}

\section{Proof of Theorems \ref{th:3} and \ref{th:3.2}}
\label{ss:6.1}  \paragraph{Proof of Theorem \ref{th:3}.} Within this
proof,
 we denote
by $c_\delta$ a positive constant depending on $\delta>0$ and on
quantities~\eqref{e:only-depe}.   Referring to the notation of the
statement of Theorem~\ref{th:3}, let us set $\uchio:=
\chi_{0}^{1}-\chi_{0}^{2}$, $ \uchi:=\chi_1 -\chi_2 $, and $\uuu:=
w_1-w_2$. The pair $ (\uchi,\uuu) $ obviously satisfies
  \begin{eqnarray}
  &&
   \uchi_t  + A(\alpha (w_1))-  A(\alpha (w_2)) =0
 \quad  \aein \ \Omega \times (0,T),
\label{unodifb}\\
&&
 \label{duedifb}
\delta \uchi_t +A \uchi + \chi_{1}^3-\chi_{2}^3  -\uchi = \uuu \quad
\aein \ \Omega \times (0,T).
\end{eqnarray}
Following the proof of~\cite[Prop.~2.1]{rossi05},
 we test \eqref{unodifb} by
$ \mathcal{N}\left(\uuu - \mo (\uuu) \right)$, \eqref{duedifb} by  $
\mathcal{N}( \uchi_t)+ \uchi$,   add the resulting equations, and
integrate over $(0,t)$, $ t \in (0,T)$. We refer to the proof
of~\cite[Prop.~2.1]{rossi05} for all the detailed computations,
leading to (cf.~\cite[(3.51)]{rossi05})
\begin{equation}
\label{basic-cont-dep}
 \begin{aligned}
\int_0^t \| \uuu \|_{H}^2    + \delta \int_0^t \|\mn (\uchi_t)
\va^2+ \delta \|\uchi(t)  \|_{H}^2
 + \int_0^t \|
\nabla \uchi\|_{H}^2
   \leq C\left(\|\uchio\|_H^2  + \int_0^t \|\uchi\|_H^{2} \right).
\end{aligned}
\end{equation}
 An easy
application of Gronwall's lemma to the function $t \mapsto
\|\uchi(t)\|_H^{2} $ entails
\begin{equation}
\label{eq:stima-cont-dep-h} \| \uchi \|_{C^0 ([0,t];H) \cap L^2
(0,t;V)}+ \| \uchi_t\|_{L^2 (0,t;V')}+
 \| \uuu \|_{L^2 (0,t;H)}\leq c_\delta  \| \uchio\|_H.
\end{equation}
Furthermore,  exploiting~\eqref{e:hyp7-bis} and the
above~\eqref{eq:stima-cont-dep-h},
 it follows from the H\"older inequality that
\begin{equation}
\label{b}
\begin{aligned}
&\| \phi(\chi_1) -\phi(\chi_2) \|_{L^2 (0,t;H)}^2 \\ & \leq C
\int_0^t \int_\Omega |\uchi|^2 \left(\chi_1^2 +\chi_2^2  +1\right)^2
\\
&\leq C \int_0^t \left(\| \chi_1  \|_{L^6 (\Omega)}^4 + \| \chi_2
\|_{L^6 (\Omega)}^4 \right) \| \uchi \|_{L^6 (\Omega)}^2 + C
\int_0^t \int_\Omega |\uchi|^2
\\
&\leq C \left(\| \chi_1 \|_{L^{\infty}(0,T;L^{6}(\Omega))}^4 + \|
\chi_2 \|_{L^{\infty}(0,T;L^{6}(\Omega))}^4 +1\right) \| \uchi
\|_{L^2 (0,t;V)}^2 \leq c_\delta \| \uchio \|_H^2.
\end{aligned}
\end{equation}
Next, we test~\eqref{duedifb}
 by $\uchi_t $ and integrate in time to obtain
 \begin{equation}
 \label{e:6.5}
\begin{aligned}
\frac\delta2\int_0^t \| \uchi_t  \|_H^2  + \frac12 \|
\nabla(\uchi(t)) \|_H^2 \leq \frac12 \|\nabla\uchio \|_{H}^2
+c_\delta \left(\int_0^t \| \uuu \|_{H}^2 +\int_0^t \|\phi(\chi_1)
-\phi(\chi_2)\|_H^2 + \int_0^t \| \uchi\|_{H}^2 \right)\,.
\end{aligned}
 \end{equation}
 In view of~\eqref{eq:stima-cont-dep-h}--\eqref{e:6.5},  we
readily infer
 the continuous dependence estimate~\eqref{contdepV}
for $\uchi$ in $C^0 ([0,t];V) \cap H^1 (0,t;H)$. Then, the estimate
for $\| \uchi \|_{L^2 (0,t;Z)}$ follows from
\eqref{eq:stima-cont-dep-h}--\eqref{e:6.5}   by a comparison
argument. \fin 
\paragraph{Proof of Theorem~\ref{th:3.2}.} Referring to the notation
of the proof of Theorem~\ref{th:3}, we again test \eqref{unodifb} by
$ \mathcal{N}\left(\uuu - \mo (\uuu) \right)$, \eqref{duedifb} by  $
\mathcal{N}( \uchi_t)+ \uchi$,   add the resulting equations, and
integrate over $(0,t)$, $ t \in (0,T)$. Developing the same
calculations as in the above lines, we  note that the chain of
inequalities \eqref{b} is now trivial, since under the present
assumptions the functions  $\chi_1$ and $\chi_2$ are estimated in
$L^\infty (0,T;Z)$ (see Proposition~\ref{prop:regularized}). On the
other hand, the following term:
\begin{equation*}
I :=\int^t_0 \int_\Omega m(\uuu)\left(\alpha(w_1) -\alpha(w_2)\right),
\end{equation*}
which was easily estimated in the proof of Theorem~\ref{th:3}, now
needs to be carefully handled
 because of the (at most) quadratic controlled growth of
$\alpha^\prime$.  Indeed, observe that
\begin{equation*}
\begin{aligned}
\vert I\vert &\le \int_0^t \Vert m(\uuu)\Vert_{L^\infty(\Omega)}
\Vert\alpha(w_1) -\alpha(w_2)\Vert_{L^1(\Omega)}
\\
&\le C\int_0^t \left(\Vert m(\uuu)\Vert_{L^1(\Omega)}\int_\Omega
(1+\vert w_1\vert^{2p} +\vert w_2\vert^{2p})\uuu\right)
\\
&\le C \int_0^t \left(\Vert \phi(\chi_1)
-\phi(\chi_2)\Vert_{L^1(\Omega)} (1+\Vert
w_1\Vert^{2p}_{L^\infty(\Omega)} +\Vert
w_2\Vert^{2p}_{L^\infty(\Omega)})\Vert \uuu\Vert_{L^1(\Omega)}
\right)
\\
&\le C \int_0^t \left(\Vert\uchi \Vert_{L^1(\Omega)} (1+\Vert
w_1\Vert^{2p}_{L^\infty(\Omega)} +\Vert
w_2\Vert^{2p}_{L^\infty(\Omega)})\Vert \uuu\Vert_{L^1(\Omega)}
\right)
\\
&\le \varrho \int_0^t\Vert \uuu\Vert^2_{H} + C_\varrho \int_0^t
(1+\Vert w_1\Vert^{4p}_{L^\infty(\Omega)} +\Vert
w_2\Vert^{4p}_{L^\infty(\Omega)}) \Vert\uchi \Vert^2_{H}
\end{aligned}
\end{equation*}
for some $\varrho \in (0,1)$ and $C_{\varrho}>0$.
 This modification gives, in place of
\eqref{basic-cont-dep},
$$
 \begin{aligned}
&(1-\varrho)\int_0^t \| \uuu \|_{H}^2    + \delta \int_0^t \|\mn
(\uchi_t) \va^2+ \delta \|\uchi(t)  \|_{H}^2
 + \int_0^t \|
\nabla \uchi\|_{H}^2\\
   &\leq C\left(\|\uchio\|_H^2  +
   \int_0^t (1+\Vert w_1\Vert^{4p}_{L^\infty(\Omega)}
+\Vert w_2\Vert^{4p}_{L^\infty(\Omega)}) \Vert\uchi
\Vert^2_{H}\right).
\end{aligned}
$$
Thus,  recalling \eqref{reg-w}, we can use Gronwall's lemma to
deduce \eqref{eq:stima-cont-dep-h}. Estimate \eqref{e:6.5} can be
obtained by arguing as in the proof of Theorem~\ref{th:3}, hence the
result.\fin


\appendix
\section{Appendix}
\noindent
We propose the following approximate system for \emph{both}
Problem~\ref{p:1} and Problem~\ref{p:2}:
\begin{align}
  & \label{1-var-better-mudelta}
   \chi_t + A (\alpha_M(w)) =0 \qquad \aein \  \Omega \times (0,T)\,,\\
  &  \label{2-var-better-mudelta}
   \delta \chi_t+ A\chi + \phi_{\app}(\chi) =w \qquad \aein \ \Omega \times (0,T)\,,
\end{align}
depending on  the parameters $\delta,\, M, \, \app >0$, where
\begin{equation}\label{approalfa}
  \alpha_M(r)=\left\{\begin{array}{lll}
   \alpha(-M) + C_1(r-M) &\quext{if }\,r<- M,
  \\
     \alpha(r)& \quext{if }\,|r|\le M,\\
     \alpha(M) + C_1(r-M) &\quext{if }\,r> M,
   \end{array}
   \right.
\end{equation}
 $C_1$ being the same constant as in \eqref{e:hyp1},
 and
\begin{equation}\label{approphi}
  \phi_{\app}(r)=
\left\{
  \begin{array}{lll}
       \phi(r) &\quext{if }\,|\phi(r)|\le\frac1{\app},\\
     \frac1{\app}\sign(r) &\quext{otherwise}.
   \end{array}
   \right.
\end{equation}
It is immediate to check that, for any choice of the approximation
parameters $M$ and $\app$, the functions $\alpha_M$ and
$\phi_{\app}$ are Lipschitz continuous on $\R$  and that
\begin{equation}
\label{e:uniformly-cpt}
\begin{aligned}
&
 \alpha_M \to \alpha \qquad \text{uniformly
on compact subsets of $\RR$ as $M \nearrow +\infty$,}
\\
&
 \phi_\mu \to \phi \qquad \text{uniformly
on compact subsets of $\dom(\phi)$ as $\mu \searrow 0$.}
\end{aligned}
\end{equation}
 Of
course, the Lipschitz constants  of $\alpha_M$ and $\phi_\mu$
explode as $M \nearrow +\infty$ and $\mu \searrow 0$, respectively.
Let us also point out that, by construction,
\begin{equation}
\label{e:coercivity} \alpha_{M}'(r) \geq C_1>0 \qquad \text{for all
$r \in \R$, $M
>0,$}
\end{equation}
which yields that the inverse $\rho_M : \R \to \R$ of $\alpha_M$ is
Lipschitz continuous, with
\begin{equation}
\label{e:impo-conse} |\rho_M(x) -\rho_M(y)| \leq \frac{1}{C_1}|x-y|
\quad \text{for all $x,y \in \R$, $M
>0.$}
\end{equation}
What is more, relying on convergence \eqref{e:uniformly-cpt} of
$\phi_\app$ to $\phi$, one can also check that, for $\mu>0$
sufficiently small (say $0<\mu\leq \mu_*$), \eqref{e:2.2.14} and
\eqref{hyp:3} hold on this approximate level as well, i.e.,
\begin{equation}
\label{e:2.2.14-approx}
\begin{aligned}
 \forall\, m \in \dom(\phi)=(a,b)\ \  \exists\, C_m,\
C_m'>0 \, : \ \ &\forall\, 0<\mu\leq \mu_* \quad  \forall\, r \in
(a-m,b-m) \\ & |\phi_\app(r+m)| \leq C_m \phi_\app(r+m)r +C_m'\,,
\end{aligned}
\end{equation}
as well as
\begin{equation}
\label{hyp:3-appr} \ \ \exists\, C>0\, : \ \ \forall\, 0<\mu\leq
\mu_* \quad  \quad \forall\, r \in (a,b) \quad
|\phi_\app(r)|^{\sigma} \leq C \left( \widehat{\phi_\app}(r)
+1\right)\,,
\end{equation}
with $\sigma \in (0,1)$ as in \eqref{hyp:3}, in particular, complying
with the compatibility condition \eqref{hyp:4}.

It was proved in~\cite[Thm.~2.1]{rossi05} that, for every
$\delta,\,  M, \, \app>0$, there exists a unique pair
$(\chimd,\wmd)$, with
\begin{equation}\label{initial-regularity}
 \begin{aligned}
  & \chimd \in L^2 (0,T;Z) \cap L^\infty (0,T;V) \cap H^1 (0,T;H),\\
  &\wmd \in L^2 (0,T;V),
 \end{aligned}
\end{equation}
solving the Cauchy problem for
system~\eqref{1-var-better-mudelta}--\eqref{2-var-better-mudelta},
supplemented with some initial datum $\chi^0 \in V$.
\paragraph{Problem $\mathbf{P}_{\delta,\app}$.} In what follows, we approximate the initial datum
$\chi_0 \in V$ in~\eqref{hyp:initial-datum} by a sequence
\begin{equation}\label{basta}
 \{\chi_{0,\mu}\} \subset H^4(\Omega) \quad \text{with} \quad \chi_{0,\mu}
  \weakto \chi_0 \ \ \text{in $V$} \ \ \text{and} \ \ \sup_{\mu>0}\|
  \widehat{\phi}_{\app}(\chi_{0,\mu})\|_{L^1(\Omega)}<+\infty
\end{equation}
(for example, we can construct $\{ \chi_{0,\mu}\}$ by
applying (twice) the
elliptic regularization procedure developed in the proof
of~\cite[Prop.~2.6]{bcgg}).

For every $\delta,\,  M, \, \app>0$, we call
$\mathbf{P}_{\delta,M,\app}$ the initial and boundary value problem
obtained by  supplementing  the PDE system
\eqref{1-var-better-mudelta}--\eqref{2-var-better-mudelta} with the
initial condition
\begin{equation}\label{app-init}
  \chi(0)=\chi_{0,\mu} \quad \text{in $H^4(\Omega)$}.
\end{equation}
In the following Section \ref{ss:a.1}, we will  prove some
further regularity of the approximate solutions. In this way, we
will  justify, on the level of the approximate Problem
$\mathbf{P}_{\delta,M,\app}$,   the estimates formally performed in
Section~\ref{s:3.1}. Hence, in Section~\ref{ss:a.2}, we will develop
the rigorous proof of Theorem~\ref{th:1} by relying on the
aforementioned estimates and by passing to the limit in Problem
$\mathbf{P}_{\delta,M, \app}$, first as $\delta\searrow 0$ for
$M,\,\app>0$ fixed, then as $M \nearrow +\infty$ for $\app>0$ fixed,
and, finally, as $\app \searrow 0$.

 Furthermore, it would be possible to give
a rigorous proof of Theorem~\ref{th:2} by passing to the limit in
Problem $\mathbf{P}_{\delta,M,\app}$ first as $M \nearrow +\infty$ for
$\app>0$ fixed, and then as $\app \searrow 0$. However, we are not
going to enter into the details of the latter procedure, which
follows the very same lines as the one for Theorem~\ref{th:1}.

\begin{notation}
\upshape In what follows, we denote by $C_{\delta,M,\app}$
various constants (which can differ from occurrence to occurrence,
even in the same line),  depending on the parameters
$\delta$, $M$, and $\app$, and such that $C_{\delta,M,\app} \nearrow
+\infty$ as either $\delta \searrow 0$, or $M \nearrow +\infty$, or
$\app \searrow 0$. The symbols  $C_{\delta,\mu}$,  $C_{M,\mu}$, and
$C_{\mu}$ have an analogous meaning.
\end{notation}

\subsection{Enhanced regularity estimates on the approximate problem}
\label{ss:a.1}

\paragraph{First estimate.}
We note that $\wmd \in L^2 (0,T;V)$ and that, since $\phi_\app$
is a Lipschitz continuous function, $\chimd \in L^\infty(0,T;V)$
(cf.~\eqref{initial-regularity}) implies $\phi_\app (\chimd)  \in
L^\infty(0,T;V)$.
 Thus, by comparison in~\eqref{2-var-better-mudelta},
 we have   $\delta \partial_t \chimd +
 A\chimd \in L^2 (0,T;V)$. Hence, testing
\eqref{2-var-better-mudelta} by $A(\partial_t \chimd)$ and using
the fact that $\chimd(0)= \chi_{0,\app} \in Z$, we deduce the estimate
\begin{equation}\label{f1}
  \|\nabla \partial_t \chimd\|_{L^2 (0,T;H)}
   + \| A \chimd\|_{L^\infty (0,T;H)}
  \leq C_{\delta,\app},
\end{equation}
whence
\begin{equation}\label{f1-bis}
  \chimd\in L^\infty (0,T;Z) \cap H^1 (0,T;V).
\end{equation}
\paragraph{Second estimate.}
Since $\alpha_M$ is Lipschitz continuous and $w \in L^2(0,T;V)$, we
have
  $\alpha_M(\wmd) \in
L^2(0,T;V)$.
Estimate~\eqref{f1} and a comparison
in~\eqref{1-var-better-mudelta} yield a bound for
$A(\alpha_M(\wmd))$ in  $L^2 (0,T;V)$, whence
\[
  \alpha_M(w) \in L^2(0,T;H^3(\Omega)) \subset L^2(0,T;W^{1,\infty}(\Omega)).
\]
 Recalling   \eqref{e:impo-conse} and using the fact that $w=\rho_M (\alpha_M(w))$,
   we readily deduce the estimate
\begin{equation}\label{regosig4}
  \|\wmd \|_{L^2(0,T;W^{1,\infty}(\Omega))} \leq C_{\delta,M,\app}.
\end{equation}
%
\paragraph{Third estimate.}
Using a parabolic regularity argument in
\eqref{2-var-better-mudelta} and relying on regularity
\eqref{basta} for the approximate initial datum $\chi_{0,\app}$,
we deduce that
\begin{equation}
\label{cf-later}
  \| \partial_t \chimd \|_{L^2 (0,T;W^{1,3+\epsilon}(\Omega))} + \| A
  \chimd\|_{L^2 (0,T;W^{1,3+\epsilon}(\Omega))} \leq C_{\delta,M,\mu},
\end{equation}
where $\epsilon>0$ is a suitable number.
More precisely, since
$\chi_{0,\mu}\in H^4(\Omega)\subset W^{3,6}(\Omega)$,
the above formula holds for any $\epsilon\in(0,3]$
(cf.~inequality \eqref{hieber-pruss} below for a
justification). Thus, by interpolation, we obtain that
$\nabla \chimd$ belongs to  $ H^{1/2}
(0,T; W^{1,3+\epsilon}(\Omega))$ and,
recalling the continuous embedding
$W^{1,3+\epsilon}(\Omega) \subset L^\infty (\Omega)$,
we conclude that
\begin{equation}\label{regosig5}
  \| \nabla\chimd\|_{L^\infty (0,T;L^\infty(\Omega))} \leq
   C_{\delta,M,\app}.
\end{equation}

\paragraph{Fourth estimate.} Notice that,
for almost all $t \in (0,T)$, the function
 $\nabla(|\wmd(t)|^p \wmd(t))= (p+1)|\wmd(t)|^p  \nabla \wmd(t)$ belongs to
 $L^2(\Omega)$, thanks to~\eqref{initial-regularity}
 and~\eqref{regosig4}. Hence, for a.a. $t \in (0,T)$, we can
test~\eqref{2-var-better-mudelta} by $|\wmd(t)|^p \wmd(t)$, which yields
\begin{equation}
\label{f3}
\begin{aligned}
\int_{\Omega} &  |\wmd(t)|^{p+2} \\ & = \int_{\Omega} \nabla
\chimd(t) \cdot \nabla(|\wmd(t)|^p \wmd(t)) + \int_{\Omega}
\phi_\app(\chimd(t)) |\wmd(t)|^p \wmd(t) + \delta \int_{\Omega}
\partial_t \chimd(t) |\wmd(t)|^p \wmd(t)
\\ &
\begin{aligned}
= (p+1)\int_{\Omega}  |\wmd(t)|^p \nabla \chimd(t) \cdot  \nabla
\wmd(t)  &  + \int_{\Omega} \phi_\app(\chimd(t)) |\wmd(t)|^p
\wmd(t) \\ & - \delta(p+1) \int_{\Omega}
\alpha_M'(\wmd(t))|\wmd(t)|^p | \nabla \wmd(t)|^2,
\end{aligned}
\end{aligned}
\end{equation}
the second equality following from
equation~\eqref{1-var-better-mudelta}. We estimate the second term
on the right-hand side of the above equality by using the bound
for $\phi_\app(\chimd)$ in $L^\infty (0,T; L^\infty (\Omega))$, due
to~\eqref{f1-bis} and the Lipschitz continuity of $\phi_\app$. We
deal with the first integral term as follows:
\begin{equation}
\label{e:f2}
\begin{aligned}
  \left|\int_{\Omega}  |\wmd(t)|^p \nabla \chimd(t) \cdot  \nabla \wmd(t)  \right|
   & \le \big\| |\wmd(t)|^{p/2}\nabla \wmd(t)\big\|_{L^2(\Omega)}
    \big\| |\wmd(t)|^{p/2} \big\|_{L^2(\Omega)} \big\|\nabla \chimd(t)\big\|_{L^\infty}
  \\ &  \le \varrho \io |w(t)|^p |\nabla w(t)|^2
    + C_{\delta,M,\mu} \int_{\Omega}| w(t) |^p
    \end{aligned}
\end{equation}
for some suitable positive  constant $\varrho$, where we have also
used~\eqref{regosig5}.  Now, recalling~\eqref{e:coercivity}, we
estimate the last summand on the right-hand side of~\eqref{f3} by
\[
-\delta(p+1) \int_{\Omega} \alpha_M'(\wmd(t))|\wmd(t)|^p | \nabla
\wmd(t)|^2 \leq -\delta(p+1)C_1 \int_{\Omega} |\wmd(t)|^p |\nabla
\wmd(t)|^2,
\]
and we move the above term to the left-hand side of~\eqref{f3}.
Then, we combine the latter inequality with~\eqref{e:f2}, in which
we choose
 $\varrho=\frac{\delta (p+1)C_1}{4}$. We thus obtain, for a.a.
 $t \in (0,T)$,
\begin{equation}\label{co31}
\int_{\Omega} |\wmd(t)|^{p+2}
   +\frac34\delta(p+1)C_1 \int_{\Omega} |\wmd(t)|^p |\nabla \wmd(t)|^2
   \le C_{\delta,M,\app} \left(\int_{\Omega}|\wmd(t)|^{p+1} + \int_{\Omega}| w(t)
   |^p\right)\,.
\end{equation}
Thus, we finally infer that
\begin{equation}\label{regosig6}
  w \in L^\infty(0,T;L^p(\Omega)) \quad \text{for all $p\in[1,\infty)$,}
\end{equation}
whence, by the Lipschitz continuity of $\alpha_M$,
\begin{equation}\label{regosig7}
  \alpha_M(w) \in L^\infty(0,T;L^p(\Omega)) \quad \text{for all $p\in[1,+\infty)$.}
\end{equation}
\subsection{Rigorous proof of Theorem~\ref{th:1}}
\label{ss:a.2}

Within this section, for all $\delta,\,\app>0$, we
will denote by $\{(\chimud,\wmud)\}_{\delta,M,\app}$
the family of solutions to Problem~$\mathbf{P}_{\delta,M,\app}$.

\medskip

%
%
\noindent \underline{First step}. For  fixed $\app,M>0 $, we pass to
the limit in Problem~$\mathbf{P}_{\delta,\app}$ as $\delta\searrow
0$. We then perform the same calculations as in Section~\ref{s:3.1}
(cf. \eqref{est:1}--\eqref{e:stima1}, \eqref{est:2-deltazero},
\eqref{est:3}--\eqref{e:stima3}). Also relying on
\eqref{e:2.2.14-approx}--\eqref{hyp:3-appr},  we conclude that
\begin{equation}\label{ff1}
 \begin{aligned}
 & \exists\, C>0\,:  \ \ \forall\, \delta,\, M, \, \mu>0 \quad
  \|\chimud\|_{L^\infty (0,T;V)} + \| \wmud\|_{L^2 (0,T;V)}
   + \| \widehat{\phi}_\app (\chimud)\|_{L^\infty (0,T;L^1 (\Omega))} \\
  & \mbox{}~~~~~~~~~~
   + \delta^{1/2} \| \partial_t \chimud \|_{L^2 (0,T;L^2(\Omega))}
   + \| (\alpha_M'(\wmud))^{1/2} \nabla \wmud \|_{L^2 (0,T;L^2(\Omega))}
   \leq C.
 \end{aligned}
\end{equation}
Recalling the definition of  $\phi_\app$ and its  Lipschitz
continuity, we also have
\begin{equation}\label{ff1-bis}
 \begin{aligned}
  \exists\, C_\mu>0\,:  \ \ \forall\, \delta,\, M>0\quad
  \| \phi_\app (\chimud)\|_{L^\infty (0,T;V)\cap L^\infty (0,T;L^\infty
  (\Omega))} \leq C_\mu.
 \end{aligned}
\end{equation}
In the same way, estimate \eqref{ff1} for $\wmud$ and  the Lipschitz
continuity of $\alpha_M$  yield
\begin{equation}\label{est-intermediate}
  \exists\, C_{M}>0\,:  \ \ \forall\, \delta,\, \app>0 \quad
  \|\alpha_M(\wmud)\|_{L^2 (0,T;V)} \leq C_{M}.
\end{equation}
 Next, a comparison in
\eqref{2-var-better} and the maximal parabolic regularity result
from~\cite{hieber-pruss} yield
\begin{equation}
\label{hieber-pruss} c(\delta)\int_0^T \| \partial_t \chimud\|_{L^6
(\Omega)}^2 + \int_0^T \| A \chimud\|_{L^6 (\Omega)}^2  \leq C
\int_0^T \| \ell_{\delta,M,\app}\|_{L^6 (\Omega)}^2,
\end{equation}
for some $c(\delta)$ such that $c(\delta)\to 0$ as $\delta \to 0$, where we have set
\[
 \ell_{\delta,M,\app} = \wmud -\phi_{\mu}(\chimud) -A\chi_{0,\app}.
\]
In view of estimates \eqref{ff1} for $\wmud$,  \eqref{ff1-bis} for
$\phi_\app (\chimud)$ in $L^\infty (0,T;V)$, and \eqref{basta} for
$\{\chi_{0,\app}\}$, we conclude that
\[
\|  \ell_{\delta,M,\app} \|_{L^2 (0,T;L^6(\Omega))} \leq C_\mu.
\]
Therefore, \eqref{hieber-pruss} gives
\begin{equation}
 \label{inghippo}
  \exists\, C_\mu>0\,: \ \ \forall\, \delta,\, M>0 \quad \|\chimud\|_{L^2
  (0,T;W^{2,6}(\Omega))} \leq C_\app.
\end{equation}
On the other hand, estimate \eqref{est-intermediate} and a
comparison in~\eqref{1-var-better-mudelta} imply
\begin{equation}
  \label{est-chit} \exists\, C_{M}>0\,:  \ \ \forall\, \delta, \, \app >0\quad
    \|\partial_t \chimud\|_{L^2 (0,T;V')} \leq C_{M}.
\end{equation}

On behalf of the above estimates and arguing in the very same way as
in Section~\ref{s:3.2}, we see that, for every fixed $M>0$ and
$\mu>0$, there exist a sequence $\delta_k \searrow 0$ (for notational
simplicity, we do not highlight its dependence on the parameters $M$
and $\mu$) and functions $(\chi_{M,\app},w_{M,\app},
\bar{\alpha}_{M,\app})$ such that the sequence
$\{(\chi_{\delta_k,M,\mu},w_{\delta_k,M,\mu}) \}$ converges to
$(\chi_{M,\app},w_{M,\app})$, as $k \to +\infty$,  in the sense
specified by~\eqref{e:conv-chi-strong}--\eqref{e:conv-chi-weak1},
\eqref{e:conv-w-strong}, as well as
\[
 \begin{aligned}
  & \partial_t \chi_{\delta_k,M,\mu} \weakto \partial_t \chi_{M,\app} \quad
   \text{in $L^2 (0,T;V')$,}\\
  & \delta_k^{1/2} \partial_t \chi_{\delta_k,M,\mu} \weakto 0 \quad
   \text{in $L^2(0,T;H)$,}\\
  & \alpha_M(w_{\delta_k,M,\mu}) \weakto   \bar{\alpha}_{M,\app} \quad
  \text{in $L^2 (0,T;V)$.}
\end{aligned}
\]
Next, arguing similarly to
the (formal) proof of Theorem~\ref{th:1}, we
conclude that
\begin{equation}
\label{post}
  \phi_\app (\chi_{\delta_k,M,\mu}) \to \phi_\app (\chi_{M,\app}) \quad
   \text{in $L^2 (0,T;H)$.}
\end{equation}
Finally,  we use \eqref{post}  in the very same way as in
Section~\ref{s:3.2} to infer
\[
\bar{\alpha}_{M,\app} =\alpha_M (w_{M,\app})
\]
and
\[
  \alpha_M(w_{\delta_k,M,\mu}) \to   \alpha_M (w_{M,\app}) \quad
   \text{in $L^2 (0,T;H)$.}
\]
Therefore, we conclude that the pair $(\chi_{M,\app},w_{M,\app})$
is a solution to the PDE system
\begin{align}
  & \label{1-var-better-mu}
   \chi_t + A (\alpha_M(w)) =0 \qquad \aein \  \Omega \times (0,T)\,,\\
 &
 \label{2-var-better-mu} A\chi + \phi_{\app}(\chi)
  = w \qquad \aein \ \Omega \times (0,T)\,,
\end{align}
supplemented with the initial condition~\eqref{app-init}.
\bigskip
\\
\underline{Second step}. We now take the limit $M\nearrow +\infty$ in
(the Cauchy problem for)
\eqref{1-var-better-mu}--\eqref{2-var-better-mu}.
Estimates~\eqref{ff1} and \eqref{ff1-bis} hold for the sequence of
solutions $\{(\chi_{M,\app},w_{M,\app})\}_M$ as well.
 Furthermore,  using a lower-semicontinuity argument, we also deduce
 from \eqref{inghippo} that
\begin{equation}
\label{bor} \|\chi_{M,\app}\|_{L^2 (0,T;W^{2,6}(\Omega))}\leq
C_\mu\quad \text{for all $M>0.$}
\end{equation}

Now, we point out that  \eqref{ff1} entails
\begin{equation}
\label{stima_g} \int_0^T \int_{\Omega} \alpha_M'(w_{M,\app})|\nabla
w_{M,\app} |^2 \leq C \quad \text{for all $M>0$}.
\end{equation}
Let us denote by $\mathcal{T}_M$ the truncation operator at
level $M$ and define
\begin{equation}
\label{def-taum} \tau_M:=\mathcal{T}_M(w_{M,\app} ):= \left\{
\begin{array}{lll}  -M  &\text{if } w_{M,\app} <-M,
\\
  w_{M,\app} &\text{if } |w_{M,\app}| \leq M,
 \\
  M & \text{if } w_{M,\app} > M,
\end{array}
\right. \qquad \text{a.e.}\ t\in \Omega \times (0,T)
\end{equation}
(to simplify, we omit the index $\mu$ in the notation for $\tau_M$).
 For later use, we also introduce $\forae\ t \in (0,T)$ the sets
\begin{equation}
\label{def-om} \begin{cases} \mathcal{A}_M:= \left\{(x,t)\in \Omega
\times (0,T)\, : \ |w_{M,\app}(x,t)| \leq M \right\},
\\
 \mathcal{O}_M:= \left\{(x,t)\in \Omega \times (0,T)\, : \
|w_{M,\app}(x,t)|> M \right\},
\\
 \mathcal{O}_M^t:= \left\{x\in \Omega\, : \ (x,t) \in \mathcal{O}_M
\right\}.
\end{cases}
\end{equation}
From \eqref{stima_g}, we also infer
\[
\int_0^T \int_{\Omega} \alpha'(\tau_M)|\nabla \tau_M |^2 \leq C
\quad \text{for all $M>0$},
\]
whence, in view of \eqref{e:hyp1},
\begin{equation}
\label{stima_h} \| |\tau_M|^p\, \nabla \tau_M \|_{L^2 (0,T;H)} \leq
C
 \quad \text{for all $M>0$}.
\end{equation}
Now, in order to reproduce estimates
\eqref{e:to-be-cited3}--\eqref{e:3.18} in the present approximate
setting, we  test \eqref{2-var-better-mu} by $|\tau_{M}(t)|^p
\tau_{M}(t)$ for a.e. $t \in (0,T)$.
Clearly,
\[
\int_{\Omega} w_{M,\app}(t)|\tau_{M}(t)|^p \tau_{M}(t)
\ge\int_{\Omega} | \tau_{M}(t)|^{p+2},
\]
so that we have (cf. also \eqref{f3})
\begin{equation}\label{stima_j}
 \begin{aligned}
  \int_{\Omega}   | \tau_{M}(t)|^{p+2}
   & \le (p+1)\int_{\Omega} | \tau_{M}(t)|^p \,\nabla  \chi_{M}(t) \cdot  \nabla \tau_{M}(t)
   + \int_{\Omega} \phi_\app (\chi_{M,\app}(t))|\tau_{M}(t)|^p \tau_{M}(t)\\
  & \leq (p+1) \| | \tau_{M}(t) |^p \,\nabla \tau_{M}(t) \|_H
     \| \nabla \chi_{M}(t)\|_H
      + C_\mu \int_{\Omega} | \tau_M(t)|^{p+1} \\
  & \leq C \||\tau_{M}(t)|^p \,\nabla \tau_{M}(t) \|_{H}
   + \frac12 \int_{\Omega} | \tau_M (t)|^{p+2}
   + C_\mu,
 \end{aligned}
\end{equation}
where the second inequality follows from estimate \eqref{ff1-bis} and
the last one from \eqref{ff1} and
Young's inequality. Therefore, combining
\eqref{stima_h} and \eqref{stima_j}, we find an estimate for $\|
\tau_M\|_{L^{p+2}(\Omega)}^{p+2}$ in $L^2 (0,T)$ with some constant
$C_\mu>0$  which is independent of $M>0$. On behalf
of~\eqref{def-taum}--\eqref{def-om}, from the latter bound, we infer
(recall that $|\cdot|$ also denotes the Lebesgue measure)
\begin{equation}
\label{stima_m}
\begin{aligned}
C_\mu \geq \int_{0}^T \left( \int_{\mathcal{O}_M^t} |M|^{p+2}\, \dd
x\right)^2 \, \dd t = M^{2p+4} \int_0^T |\mathcal{O}_M^t|^2 \, \dd t
\geq \frac{ M^{2p+4}}{T}|\mathcal{O}_M|^2  \quad \text{for all
$M>0$},
\end{aligned}
\end{equation} where the last inequality is a direct consequence of Jensen's
inequality. Next, we apply the nonlinear Poincar\'e inequality
\eqref{e:poinc-gen} to $|\tau_M|^p \tau_M$, thus obtaining
(cf.~\eqref{e:poinc-est})
\[
 \||\tau_M|^p \tau_M  \|_{V} \leq K \left( \| \nabla
(|\tau_M|^{p} \tau_M) \|_{H} + \left|
\mo(\tau_M)\right|^{p+1}\right)\,.
\]
In view of \eqref{stima_h} and of the definition of $\tau_M$, we
find an estimate for $|\tau_M|^p \tau_M $ in $L^2 (0,T;V)$, again
with some constant $C_\app$ which is independent of $M>0$. Hence, using the fact that
$V \subset L^6 (\Omega)$ and the growth condition \eqref{e:hyp1} for
$\alpha$, we conclude that
\begin{equation}
\label{stima_l} \| \alpha(\tau_M)\|_{L^{\rho_p}(0,T; L^{\kappa_p}
(\Omega))} \leq C_\mu  \quad \text{for all $M>0$}
\end{equation}
(where the indexes
$\rho_p$ and $\kappa_p$ are as in \eqref{e:index-notation}: in
particular, $1<\rho_p<2$). Therefore, we have
\[
\begin{aligned}
\int_0^T \int_{\Omega} |\alpha_M (w_{M,\app})|^{\rho_p}  & \leq
\iint_{\mathcal{A}_M} |\alpha (\tau_{M})|^{\rho_p} + 2^{\rho_p-1}
\iint_{\mathcal{O}_M}|\alpha(M)|^{\rho_p}
 + 2^{\rho_p-1} C_1^{\rho_p}
\iint_{\mathcal{O}_M}| w_{M,\app} - M|^{\rho_p}\\
&  \leq
 2^{\rho_p-1} \int_0^T \int_{\Omega} |\alpha(\tau_M)|^{\rho_p}+
C\| w_{M,\app}\|_{L^{\rho_p}(0,T; L^{\rho_p}(\Omega))}^{\rho_p} +
 CM^{\rho_p} |\mathcal{O}_M| \\
 & \leq C_\mu + C + C \frac{M^{\rho_p}}{M^{p+2}},
 \end{aligned}
\]
where the first inequality follows from  the very definition
\eqref{approalfa} of $\alpha_M$, the second one from trivial
calculations, and the last one from estimates \eqref{ff1} for
$w_{M,\app}$, \eqref{stima_m} for $|\mathcal{O}_M|$, and
\eqref{stima_l} for $\alpha(\tau_M)$. Note that, since $\rho_p<2$,
we have
 $M^{\rho_p}/ M^{p+2} \to 0 $ as $M \to +\infty$.

Altogether, we find
\begin{equation}
\label{stima_n} \|\alpha_M (w_{M,\app})\|_{L^{\rho_p}(0,T;
L^{\rho_p}(\Omega))} \leq C_\mu \quad \text{for all $M>0$,}
\end{equation}
 which yields, by comparison in \eqref{1-var-better-mu},
\begin{equation}
\label{stima_p} \|
\partial_t \chi_{M,\app} \|_{L^{\rho_p} (0,T;W^{-2,\rho_p}(\Omega))}
 \leq C_\mu \quad \text{for all $M>0$,}
\end{equation}
$W^{-2,\rho_p}(\Omega)$ denoting here the standard negative order
Sobolev space.

Collecting estimates \eqref{ff1}, \eqref{ff1-bis}, \eqref{bor}, and
\eqref{stima_n}--\eqref{stima_p}, we then argue in the same way as
in Section \ref{s:3.2}. Thus,  we conclude that there exist a
subsequence $M_k \nearrow +\infty$  as $k \to +\infty$ (whose
dependence on the index $\mu>0$ is not highlighted) and functions
$(\chi_\app,w_\app)$ fulfilling~\eqref{reg-chi}--\eqref{reg-w0} such
that  the functions $(\chi_{M_k,\app}, w_{M_k,\app})$ converge, as
$k \to +\infty,$ to $(\chi_\app,w_\app)$ in the same sense as
in \eqref{e:conv-chi-strong}-\eqref{e:conv-chi-weak1} and
\eqref{e:conv-w-strong}, while, in place of
\eqref{e:conv-chi-weak2}, we only have
\[
 \dt\chi_{M_k,\app}
  \weakto \dt\chi_{\app} \quad
   \text{in \ $L^{\rho_p} (0,T;W^{-2,\rho_p}(\Omega))$,}
\]
which is, anyway, sufficient for what follows.
Furthermore,
there exists some $\bar{\alpha}_\mu \in L^{\rho_p}(0,T;
L^{\rho_p}(\Omega))$ such that
\[
\alpha_M (w_{M_k,\app}) \weakto \bar{\alpha}_\mu  \quad \text{in
$L^{\rho_p}(0,T; L^{\rho_p}(\Omega))$.}
\]
Again, we prove that
\begin{equation}
\label{post-2}
  \phi_\app (\chi_{M_k,\mu}) \to \phi_\app (\chi_{\app}) \quad
   \text{in $L^2 (0,T;H)$}
\end{equation}
and, proceeding as in Section~\ref{s:3.2}, with \eqref{post-2} we
show that $\bar{\alpha}_\mu  =\alpha(w_\mu)$ and
\[
\alpha_{M_k}(w_{M_k,\mu}) \to \alpha(w_\app) \qquad \text{in $L^1
(0,T; L^1 (\Omega))$.}
\]
Having this, we  conclude that the pair $(\chi_\app,w_\app)$ is
solution to the PDE system
\begin{align}
  & \label{1-var-better-mu-bis}
   \chi_t + A (\alpha(w)) =0 \qquad \aein \  \Omega \times (0,T)\,,\\
 &
 \label{2-var-better-mu-bis} A\chi + \phi_{\app}(\chi)
  = w \qquad \aein \ \Omega \times (0,T)\,,
\end{align}
supplemented with the initial condition~\eqref{app-init}.

\medskip

\noindent%
\underline{Third step}. Finally, we take the limit $\mu\searrow 0$
in (the Cauchy problem for)
\eqref{1-var-better-mu-bis}--\eqref{2-var-better-mu-bis}.
Estimate~\eqref{ff1}, with $\alpha_M$ replaced by $\alpha$,
 holds for the sequence
$\{(\chi_\app,w_\app)\}_\app$ for a constant $C>0$
which is \emph{independent} of the parameter $\mu>0$.

 Furthermore, using the fact that system
\eqref{1-var-better-mu-bis}--\eqref{2-var-better-mu-bis} has  the
same structure as~\eqref{1-var-better}--\eqref{2-var-better}, we
argue as in \eqref{e:more-convenient_form}--\eqref{e:stima3-bis} and
 conclude that
\[
\exists\, C>0 \ \ \forall\, \mu>0\, : \quad \|\chi_{\app}\|_{L^2
(0,T;W^{2,6}(\Omega))} +
 \| \phi_\app
(\chi_\app)\|_{L^2 (0,T; L^6(\Omega))}  \leq C.
\]
 From the bound for
 $(\alpha'(w_\app))^{1/2} \nabla
w_\app$ in $ L^2(0,T;H)$
(which follows from \eqref{stima_g} by applying
once more Ioffe's theorem), developing the very same calculations as
throughout~\eqref{e:crucial0}--\eqref{e:3.18}, we find
\begin{equation}\label{giulio11}
\begin{aligned}
 \exists\, C>0 \ \ \forall\, \mu>0\,: \quad \|
\partial_t \chi_\app \|_{L^{\eta_{p\sigma}}} (0,T;
\sobneg{-2}{{\kappa_p}}) &  +
 \| \alpha(w_\app) \|_{L^{\eta_{p\sigma}} (0,T; L^{\kappa_p}
 (\Omega))} \\ & +  \| \phi_\app(\chi_\app) \|_{L^{\sigma q_\sigma} (0,T;L^{6}(\Omega))} \leq
 C
 \end{aligned}
\end{equation}
(where the indexes $\eta_{p\sigma}$ and $q_\sigma$ are as in
\eqref{e:index-notation} and \eqref{e:crucial}, respectively).

Thanks to the above estimates, we conclude that  there exist a
vanishing sequence $\app_k \searrow 0$ and functions $(\chi,w)$
satisfying~\eqref{reg-chi}--\eqref{reg-w0} such that
$(\chi_{\app_k}, w_{\app_k})$ converges to $(\chi,w)$ in the
topologies of~\eqref{e:conv-chi-strong}--\eqref{e:conv-alpha-weak}.
We then pass to the limit  as $k \to +\infty$ in~\eqref{basta} and,
also in view of~\eqref{app-init}, infer that $\chi$ complies with
the initial condition~\eqref{e:init-cond}.
 Furthermore, we deduce
 from the strong convergence of $\chi_{\mu_k}$ to $\chi$ in $L^2
(0,T;H)$ that $\chi_{\mu_k} \to \chi$ almost everywhere in $\Omega
\times (0,T)$. Using the uniform convergence~\eqref{e:uniformly-cpt}
of $\{\phi_{\mu_k}\}$ to $\phi$, we infer that
\[
 \phi_{\mu_k} (\chi_{\mu_k}(x,t)) \to \phi(\chi(x,t)) \quad \foraa\, (x,t)
 \in \Omega \times (0,T).
\]
Then, taking into account
 the uniform integrability of $\{ \phi_\app
(\chi_{\mu_k})\}$ in $L^2 (0,T;H)$ (which follows from
\eqref{giulio11}, noting that $\sigma q_\sigma>2$), in view of
Theorem~\ref{t:ds} we obtain
 \begin{equation}
\label{step-fundamental}
  \phi_{\mu_k}(\chi_{\mu_k}) \to
\phi(\chi) \qquad \text{in $L^2 (0,T;H)$.}
\end{equation}
  Then,  we again argue as
in Section~\ref{s:3.2} and use \eqref{step-fundamental}
 to  prove that
\[
\alpha(w_{\mu_k}) \to \alpha(w) \qquad \text{in $L^1 (0,T; L^1
(\Omega))$.}
\]
Having this, we  conclude that the pair $(\chi,w)$ is solution to
Problem~\ref{p:1}, which finishes the proof. \fin



\end{document}